\documentclass[11pt,reqno]{amsart}
\usepackage{amsmath, amssymb, amsthm,color,longtable,datetime}
\usepackage{color,amsmath,amsfonts,amssymb,amsthm,latexsym,marginnote,graphicx,mathabx}

\usepackage{hyperref}

\usepackage{setspace}
\usepackage{array}

\usepackage[margin=1.1in]{geometry}

\numberwithin{equation}{section}

\theoremstyle{plain}

\numberwithin{equation}{section}
\newtheorem{thm}{Theorem}[section]
\newtheorem{prop}[thm]{Proposition}
\newtheorem{lemma}[thm]{Lemma}
\newtheorem{cor}[thm]{Corollary}

\theoremstyle{remark}
\newtheorem{remark}[thm]{Remark}

\theoremstyle{definition}

\newtheorem*{acknowledgements}{Acknowledgements}

\providecommand{\abs}[1]{\left\lvert#1\right\rvert}
\providecommand{\norm}[1]{\left\|#1\right\|}
\providecommand{\ip}[1]{\langle#1\rangle}

\renewcommand{\div}{{\rm div}}
\newcommand{\curl}{{\rm curl}}

\newcommand{\Z}{\mathbb{Z}}

\newcommand{\N}{\mathbb{N}}

\newcommand{\R}{\mathbb{R}}

\newcommand{\supp}{\rm supp }
\newcommand{\wt}{\widetilde}

\newcommand{\what}{\widehat}

\newcommand{\ol}{\overline}

\newcommand{\e}{\varepsilon}

\newcommand{\ul}[1]{\underline{\smash{#1}}}

\title{On the Global Stability of a Beta-Plane Equation}
\author{Fabio Pusateri}
\author{Klaus Widmayer}
\date{\today}

\address{Department of Mathematics, Princeton University, Princeton, NJ 08544, USA}
\email{fabiop@math.princeton.edu}

\address{Department of Mathematics, Brown University, 151 Thayer Street, Providence, RI 02912, USA}
\email{klaus@math.brown.edu}

\begin{document}

\begin{abstract}
We study the motion of an incompressible, inviscid two-dimensional fluid in a rotating frame of reference.
There the fluid experiences a Coriolis force, which we assume to be linearly dependent on one of the coordinates.
This is a common approximation in geophysical fluid dynamics and is referred to as $\beta$-plane.
In vorticity formulation the model we consider is then given by the Euler 
equation with the addition of a linear anisotropic, non-degenerate, dispersive term.
This allows us to treat the problem as a quasilinear dispersive equation whose linear solutions exhibit decay in time at a critical rate. 

Our main result is the global stability and decay to equilibrium of sufficiently small and localized solutions.
Key aspects of the proof are the exploitation of a ``double null form'' that 
annihilates interactions between spatially coherent waves 
and a lemma for Fourier integral operators which allows us to control a strong weighted norm. 
\end{abstract}

\maketitle
\tableofcontents

\section{Introduction}
A basic model for a fluid in a rotating frame of reference is given by the Euler-Coriolis equation
\begin{equation}
\label{NSC}
 \begin{cases}
  \begin{array}{rl}
  \partial_t v+v\cdot\nabla v + f\Omega\wedge v+\nabla p = 0, & \\ \div\; v = 0, &
  \end{array}
 \end{cases}
\end{equation}
where $v = (v_1,v_2,v_3) :(t,x) \in \R \times\R^3\to\R^3$ and $p:(t,x) \in \R\times\R^3\to\R$ are the velocity and pressure of the fluid, respectively.
Here, $f\Omega \wedge v$ is the Coriolis force experienced in the rotating frame, with $\Omega\in\R^3$ being the axis of rotation and $f:\R^3\to\R$ 
the strength of the effect, which depends on the spatial location (but not on time).
To describe waves on the surface of the Earth, a common approximation in geophysical fluid dynamics (see \cite{mcwilliamsGFD,GFD})
consists in choosing $\Omega=(0,0,1)^\intercal$ and assuming trivial dynamics in the vertical direction, i.e.\ $\partial_3 v=0$.
One can then reduce matters to a two-dimensional system
\begin{equation}
\label{equ}
 \begin{cases}
  \begin{array}{rl}
  \partial_t u + u\cdot\nabla u + {(-f u_2,f u_1)}^\intercal + \nabla p = 0, & \\ \div\; u = 0, &
  \end{array}
 \end{cases}
\end{equation}
where now $u:(t,x) \in \R\times\R^2\to\R^2$, $p:(t,x) \in \R\times\R^2\to\R$ and $f:\R^2\to\R$.
A solution to the original system \eqref{NSC} is then recovered by setting $(v_1,v_2) = (u_1,u_2)$ and solving a transport equation for $v_3$.

Passing to a scalar equation using the vorticity $\omega:=\curl\; u=\partial_1 u_2-\partial_2 u_1$ yields
\begin{equation}
\label{eqo}
 \partial_t \omega+u\cdot\nabla\omega=-u\cdot\nabla f,\qquad u=\nabla^\perp(-\Delta)^{-1}\omega.
\end{equation}
On a rotating sphere, such as the Earth, the force $f$ varies with the sine of the latitude.
In a first rough approximation, so-called $f$-plane approximation, this variation is ignored,
and a fixed value $f_0$ 
is used throughout the domain.
A more accurate and very common\footnote{Such a modeling assumption is made in various contexts: examples include rotating shallow-water equations, Rossby waves and quasi-geostrophic scenarios, see \cite[Chapter 4]{mcwilliamsGFD}, \cite[Chapter 4]{Majda-A-O-PDE}, \cite[Chapter 3]{GFD} amongst others. We also remark that in \cite{AIP_1.3141499}, equation \eqref{eq:betaplane} was viewed as part of a larger family of equations to model 2d dispersive turbulence.} model in geophysical fluid dynamics is a linear approximation to this variability, which is usually referred to as ``$\beta$-plane'', 
see e.g.\ \cite[Chapter 2]{mcwilliamsGFD}, \cite[Chapter 3]{GFD}.
Assuming that the strength of the Coriolis force depends linearly on the latitude,
\begin{equation*}
 f(x,y)=f_0+\beta (y-y_0),
\end{equation*}
we arrive at the so called \emph{$\beta$-plane equation}
\begin{equation}\label{eq:betaplane}
\partial_t \omega+u\cdot\nabla\omega = \beta L_1\omega, \qquad L_1 := \frac{\partial_x}{\Delta} =\frac{R_1}{\abs{\nabla}}, 
  \qquad u=\nabla^\perp(-\Delta)^{-1}\omega,
\end{equation}
for $\omega:\R\times\R^2\to\R$.
Here $\beta$ is the parameter of linearity of the Coriolis force, which by rescaling can be assumed to be equal to one, 
and $R_1$ stands for the Riesz transform in the first coordinate:
\begin{align*}
\what{R_1 g}(\xi) = \frac{-i\xi_1}{|\xi|} \what{g}(\xi), \qquad \what{g}(\xi) := \frac{1}{2\pi} \int_{\R^2} e^{-ix \cdot \xi} g(x) \, dx.
\end{align*}

On one hand, one can view \eqref{eq:betaplane} as a perturbation of the Euler equation by a constant coefficient differential operator
and show, by arguments akin to those for 2d Euler, 
the existence of global solutions (even for large data)
with at most double exponential growth in $H^s$, $s>1$ (see \cite[Appendix B]{b-plane}).
On the other hand \eqref{eq:betaplane} can also be viewed as a quasilinear dispersive equation,
in the sense that it is a nonlinear version of the equation $\partial_t\omega=L_1\omega$, solutions of which exhibit dispersive decay 
as will be shown further below.

\subsection{Main Result}
The content of this article is a treatment of the nonlinear problem \eqref{eq:betaplane}, with the result that for sufficiently small and localized initial data, 
solutions to the Cauchy problem decay like solutions of the linear problem, and the zero solution of \eqref{eq:betaplane} is globally nonlinearly stable in a strong sense.
We can state our main result as follows:

\begin{thm}
\label{thm:maintheo0}
Consider the initial value problem for the $\beta$-plane equation
 \begin{equation}
\label{betaplane}
  \begin{cases}
  \begin{array}{ll}
   \partial_t \omega+u\cdot\nabla\omega = L_1\omega, & \, \quad u=\nabla^\perp(-\Delta)^{-1}\omega, \\
   \omega(0) = \omega_0. &
  \end{array}
  \end{cases}
 \end{equation}
There exist $N\gg 1$, $\varepsilon_0>0$, and a weighted $L^2$-based function space $X \subset \dot{W}^{1,1}$ on $\R^2$ such that
for any initial data with $\norm{\omega_0}_X, \; \norm{\omega_0}_{H^N}\leq \varepsilon_0$,
there exists a unique global solution of \eqref{betaplane} which decays at the linear rate, 
namely $\norm{\omega(t)}_{L^\infty}\lesssim\varepsilon_0 (1+|t|)^{-1}$, and scatters.
\end{thm}

A more precise statement of the theorem is presented as Theorem \ref{thm:maintheo} in Section \ref{sec:setup}, where we also illustrate its proof through a bootstrap argument in Subsection \ref{ssec:bootstrap}. The key difficulty here lies in establishing a global control over
a suitably chosen weighted $X$-norm of the profile of $\omega$ - see \eqref{eq:X-norm} on page \pageref{eq:X-norm} for the precise definition -
which has to be strong enough to guarantee the $L^\infty$ decay.

\subsection{Background}
To give some context we now present some of the key difficulties in treating the $\beta$-plane equation as a quasilinear dispersive equation.
%
%
The present model features a quadratic nonlinearity and a critical decay rate of $|t|^{-1}$ at the linear level.
This situation is common to many other dispersive and hyperbolic equations and a variety of different behaviors can occur 
even for small and Schwartz initial data.
For example, one could have global solutions with linear behavior as in the case of (quasilinear) wave equations \cite{K1}
with a null condition, blow-up at time $T \approx e^{1/\varepsilon_0}$ as in the compressible Euler equations \cite{Sideris},
nonlinear asymptotics in the sense of modified scattering as for nonlinear Schr\"odinger equations \cite{HN,KP}, 
or growth at infinity as in \cite{Alinhac}.

In the present case solutions are already known to be global, so no blow-up occurs.
Moreover, one can notice that there is a null structure in \eqref{betaplane}.
More precisely, since $u = \nabla^\perp (-\Delta)^{-1} \omega$, the transport term $u \cdot \nabla \omega$ is 
depleted when two parallel frequencies interact.
On the negative side one should also notice that, when seen as a bilinear term in $\omega$, the nonlinearity is singular
because of the $(-\Delta)^{-1}$ factor.
Moreover, the linear operator $L_1$ is anisotropic, and the impossibility of commuting the equation with rotations
introduces several difficulties.

\subsubsection*{Inviscid Euler and the Role of Dispersion}
Generally, inviscid Euler-type nonlinearities can lead to double exponential growth,
as was shown by the example of Sverak and Kiselev \cite{KiselevSverak} on a bounded domain;
see also the works of Denisov \cite{Dengrowth} and Zlato\v{s} \cite{Zlagrowth}.
In the whole space the question of global stability and asymptotic behavior for the Euler equation is widely open. 
A byproduct of Theorem \ref{thm:maintheo} is that
for sufficiently small data instability in \eqref{betaplane} is prevented by dispersion:
waves with different frequencies travel with distinct velocities and their interactions lose strength over time. 
However, this is a much weaker effect than damping or friction.
Indeed for \eqref{betaplane} the same $L^2$ based estimates as for the inviscid Euler equation $\partial_t\omega+u\cdot\nabla\omega=0$ hold,
because of the skew symmetry (for the inner product in $L^2$) 
of the constant coefficient right-hand side operator $L_1$.
Also, all Sobolev norms are preserved by the linear flow, and the same blow-up criterion as for 2d Euler holds.

As is shown in this article, \emph{the dispersion produced by $L_1$ acts as a regularizing mechanism that globally stabilizes the fluid}.
A first way of seeing improvements at the hands of dispersion is through a basic energy estimate yielding the following:
assuming a linear decay rate of $\abs{t}^{-1}$ for $Du$ in $L^\infty$
one obtains the slow growth of all Sobolev norms for the nonlinear problem 
(whereas in the absence of dispersion, or without control on the rate of dispersion, the best known bounds are double exponential -- 
see \cite[Appendix B]{b-plane}).
A finer understanding of the interactions 
in the Euler-type nonlinearity is then needed to show that decay occurs for nonlinear solutions.

In earlier work of T. Elgindi and the second author \cite{b-plane},
stability for the $\beta$-plane equation \eqref{betaplane} for arbitrarily large times was established:
it was shown that for any $M\in\N$ there exists a threshold $\e_M>0$, below which initial data of size $\e\leq\e_M$
lead to solutions that decay on time scales at least $\e^{-M}$ -- for more details see \cite[Theorem 2.1]{b-plane}.
Apart from this work, the literature on the $\beta$-plane equation is oriented towards questions of relevance in the realm of geophysical fluid dynamics.
An exhaustive list is beyond the scope of this article, and beyond the expertise of its authors, so we refer the 
reader for some overview to the books \cite{MR1925398,Majda-A-O-PDE,mcwilliamsGFD}, for example.



\subsubsection*{Resonance Structure and (Double) Null Form}
At the basis of our approach is the formulation of the problem in a way that makes it amenable 
to techniques from harmonic analysis.
This is done by working with the profile of the vorticity $f(t):=e^{-tL_1}\omega(t)$, and writing the Duhamel formula for solutions of \eqref{betaplane}
in terms of this profile $f$ in Fourier space, so to obtain an integral expression which can be viewed as an oscillatory integral --
see the beginning of Section \ref{sec:setup} and the formulas \eqref{defhatf}-\eqref{defPhi}.

From this point of view the resonances of the equation, that is, roughly speaking, those sets of frequencies that do not produce oscillations,
play a key role in the analysis of the nonlinear interactions.
This starting point is inspired by the method of space-time resonances, as introduced in 
\cite{GMSGWW3d}.
Without entering into too much detail, for now we point out that the space-time resonant set for this equation is one dimensional,
which is the generic situation for quadratic nonlinearities in two dimensions;
thus it does not provide any additional smallness, in contrast to other problems such as \cite{GMSGWW3d,GMS2dQNLS}.
However, as already pointed out above, a null form is available in the nonlinearity:
the symbol of the quadratic interaction, see \eqref{defhatf}-\eqref{defPhi}, 
vanishes when $\nabla_\eta \Phi$ does, see \eqref{eq:nabla_etaPhi}.
See also the models in \cite{PusateriShatah,OhPusateri,HPS} for similar behaviors. 

In fact, as we shall explain in detail below, even more is true for \eqref{betaplane}: One has a ``double'' null form, a quadratic
(instead of linear) degree of vanishing of the symbol, as can be seen by symmetrizing the expression \eqref{defhatf}.
This is a key insight which greatly improves the control one has over interactions close to the (space) resonant set, and for example
yields much better decay estimates for $\partial_t f$ than one would normally expect.

In our proof we will also exploit the special, anisotropic, geometric structure of interactions 
near the (time) resonances through a $TT^*$ argument,
which was previously used in \cite{DIPP,DIPau}.
However, here we employ such an argument in a different context,
not for the purpose of establishing energy estimates,
but as another means of extracting more oscillations in the bilinear interactions.
This allows us to prove a strong weighted bound for our solutions which in turn implies the desired decay over time.

\subsection{Plan of the Article}
In Section \ref{sec:setup} we begin by setting up the problem and give our detailed functional framework.
We then state a precise formulation of Theorem \ref{thm:maintheo0} (see Theorem \ref{thm:maintheo})
and discuss its proof using a bootstrap argument.
We see there that a fractional weighted estimate, see \eqref{MainBound100'}, 
is at the core of our efforts.
By symmetrizing the formulation of the $\beta$-plane equation we obtain a ``double null form''.
As a first application this yields improved bounds for the first iterate (see Lemma \ref{Lemdsf}).
The rest of the article is then devoted to establishing the weighted estimate.

In Section \ref{sec:prelim+FS} we go through preliminary reductions and a finite speed of propagation argument
that limits the range of parameters we need to consider for the weighted estimate.
Further reductions are then presented in Section \ref{sec:weightI}.
Using various localizations we balance smallness of relevant sets and repeated integration by parts to essentially reduce
to a problem where only frequencies of roughly order $1$ are involved.
These arguments crucially rely on the improved bounds due to the double null form achieved through symmetrization.

Finally, in Section \ref{sec:weightII} we exploit a non-degeneracy property of the phase function $\Phi$ (defined in \eqref{defhatf}-\eqref{defPhi}) via a $TT^*$ argument,
in combination with an appropriate anisotropic localization, thereby concluding the proof of the weighted estimate.

In Section \ref{sec:aux} we collect some useful lemmata.

\begin{acknowledgements}
The authors would like to thank Tarek Elgindi for his helpful comments in many joint discussions.
\end{acknowledgements}

\section{Setup}\label{sec:setup}
The Duhamel formulation associated to the $\beta$-plane equation \eqref{betaplane} is
\begin{equation*}
 \omega(t)=e^{tL_1}\omega_0+\int_0^t e^{(t-s)L_1}u\cdot\nabla\omega(s)\;ds.
\end{equation*}
Written in terms of the profile
\begin{align*}
f(t):=e^{-tL_1}\omega(t)
\end{align*}
this reads
\begin{equation}\label{defhatf}
 \hat{f}(t,\xi)=\hat{f}_0(\xi) + \frac{1}{(2\pi)^2}\int_0^t\int_{\R^2} 
  e^{is\Phi(\xi,\eta)}\frac{\xi\cdot\eta^\perp}{\abs{\eta}^2} \what{f}(s,\xi-\eta) \what{f}(s,\eta) d\eta ds
\end{equation}
with
\begin{equation}\label{defPhi}
 \Phi(\xi,\eta):=\frac{\xi_1}{\abs{\xi}^2}-\frac{\xi_1-\eta_1}{\abs{\xi-\eta}^2}-\frac{\eta_1}{\abs{\eta}^2}.
\end{equation}
From now on we will omit the time dependence of the profiles in this expression, since it is clear from the context.

We define the quadratic nonlinearity $B(f,f)$ through its Fourier transform
\begin{equation}\label{FSB}
 \mathcal{F} B(f,f) (t,\xi) := \int_0^t\int_{\R^2} e^{is\Phi(\xi,\eta)}\frac{\xi\cdot\eta^\perp}{\abs{\eta}^2}\widehat{f}(s,\xi-\eta)\widehat{f}(s,\eta) d\eta ds,
\end{equation}
so that the Duhamel formula \eqref{defhatf} can be written as
\begin{equation}\label{eq:DHFT}
\widehat{f}(t,\xi) = \widehat{f_0}(\xi) + \frac{1}{(2\pi)^2} \mathcal{F} B(f,f) (t,\xi).
\end{equation}

\subsubsection*{Conserved Quantities}
For future reference we note that an explicit calculation using \eqref{equ} and \eqref{eqo} shows that
the $L^2$-norms of both $u$ and $\omega$ are conserved along the flow of the equation:
\begin{equation*}
 \norm{\omega(t)}_{L^2} = \norm{\omega(0)}_{L^2} \quad \text{ and } \quad \norm{u(t)}_{L^2}=\norm{u(0)}_{L^2}, \qquad t \in \R.
\end{equation*}
As an immediate consequence we obtain that the $\dot{H}^{-1}$ norms of $\omega$ and $f$ are controlled as well:
\begin{equation}
\label{eq:L2cons'}
\big\| \abs{\nabla}^{-1} f \big\|_{L^2} = \big\| \abs{\nabla}^{-1} \omega \big\|_{L^2} \lesssim \norm{u}_{L^2}.
\end{equation}

\subsubsection*{Notation}
In this article we will work with localizations in frequency, space and time. 
To define them, as is standard in Littlewood-Paley theory
we let $\varphi: \R \to [0,1]$ be an even, smooth function supported in $[-8/5,8/5]$ and equal to $1$ on $[-5/4,5/4]$.
With a slight abuse of notation we also let $\varphi$ be the corresponding radial function on $\R^2$.
For $k\in\Z$ we define $\varphi_k(x) := \varphi(2^{-k}|x|) - \varphi(2^{-k+1}|x|)$, so that the family $(\varphi_k)_{k\in\Z}$
forms a partition of unity,
\begin{equation*}
 \sum_{k\in\Z}\varphi_k(\xi)=1, \quad \xi \neq 0.
\end{equation*}
We also let
\begin{align*}
\varphi_{I}(x) := \sum_{k \in I \cap \Z}\varphi_k, \quad \text{for any} \quad I \subset \R, \qquad
\varphi_{\leq a}(x) := \varphi_{(-\infty,a]}(x), \qquad \varphi_{> a}(x) = \varphi_{(a,\infty]}(x),
\end{align*}
with similar definitions for $\varphi_{< a},\varphi_{\geq a}$.
To these cut-offs we associate frequency projections $P_k$ through
\begin{equation*}
 P_k g:=\mathcal{F}^{-1}\left(\varphi_k(\xi)\what{g}(\xi)\right)
\end{equation*}
and define similarly $P_{I}g:=\mathcal{F}^{-1}\left(\varphi_{I}(\xi)\what{g}(\xi)\right)$,
$P_{\leq k}g:=\mathcal{F}^{-1}\left(\varphi_{\leq k}(\xi) \what{g}(\xi)\right)$, $k\in\Z$ etc.
We will also sometimes denote $\wt{\varphi}_k = \varphi_{[k-2,k+2]}$.

To simultaneously localize in space, for $(k,j)\in\mathcal{J}:=\{(k,j)\in \Z \times \Z: \; k+j\geq 0,\;j\geq 0\}$ we let
\begin{equation}
\label{phijk}
\varphi_j^{(k)}(x):=
\begin{cases}
 \varphi_j(x), \; & j \geq -k+1, \text{ or }j\geq 1,
 \\ \varphi_{\leq 0}(x), \; & j = 0, \quad (k \geq 0)
 \\ \varphi_{\leq -k}(x), \; & j = -k, \quad (k \leq 0).
\end{cases}
\end{equation}
Notice that for any $k\in\Z$ we have $\sum_{j \geq -\min\{0,k\}} \varphi_j^{(k)}(x) = 1$.
We then define
\begin{equation*}
 Q_{jk} g := P_{[k-2,k+2]}\varphi_j^{(k)} P_k g
\end{equation*}
to be the operator that localizes both in frequency and space. This will often be used to decompose our profiles into atoms
\begin{align}
\label{Qjkdecomp}
g = \sum_{(k,j)\in\mathcal{J}} Q_{jk} g.  
\end{align}

For notational convenience we also introduce the shorthand $\ip{t}:=\sqrt{1+t^2}$ for $t\in\R$.


\subsubsection*{The Main Norm}
Apart from the usual Sobolev and Lebesgue spaces we will be using a weighted function space built on $L^2$ in an atomic way: 
with the notation $k^+:=\max\{k,0\}$ we let
\begin{equation}\label{eq:X-norm}
 \norm{g(t)}_X:= \sup_{(k,j)\in\mathcal{J}} 2^{(k+j)(1+\delta)} 2^{4k^+} \norm{Q_{jk}g(t)}_{L^2}, \quad \delta = 0.5 \cdot 10^{-4}.
\end{equation}
This choice of norm is motivated by our quest to control the $L^\infty$ decay of $\omega$ through the 
dispersive estimate \eqref{eq:dispest} below.
The use of weighted $L^2$ norms in quasilinear dispersive problems is fairly standard.
Here we have decided to use a 
fractional weight following the functional framework introduced in \cite{IoPau1}.
The particular choice of putting the same number of derivatives (the power of $2^k$)
as the number of weights (the power of $2^j$) is dictated by the characteristics of this specific problem,
including the singularity of the bilinear form in \eqref{FSB} and the ``speed of propagation'' of linear frequencies. 

\subsubsection*{Dispersive Estimate}
For the linear semigroup $e^{tL_1}$ we have the following decay estimate:
\begin{lemma}
For $g\in\mathcal{S}(\R^2)$ and $k\in\Z$ we have
\begin{equation}\label{eq:dispest}
 \norm{e^{tL_1} P_k g}_{L^\infty} \lesssim |t|^{-1} 2^{3k} \norm{P_k g}_{L^1}.
\end{equation}
\end{lemma}
Since the Hessian of the exponent $\xi_1\abs{\xi}^{-2}$ on the Fourier side is $4|\xi|^{-6}$, and so in particular is non-degenerate,
the proof is a standard application of the stationary phase lemma -- see \cite[Proposition 4.1]{b-plane}.
We remark that the right hand side of \eqref{eq:dispest} is controlled by the $X$-norm of $g$ in \eqref{eq:X-norm} above.

\subsubsection*{Main Theorem}
In more detail, our Main Theorem \ref{thm:maintheo0} is:
\begin{thm}
\label{thm:maintheo}
Let\footnote{We did optimize on the value of $\delta$, and the related size of $N$, to make the proof more readable. 
Especially in the last part of the argument, in Sections \ref{sec:weightI} and \ref{sec:weightII},
improvements on this values would be possible by tracking more carefully the various parameters involved,
but due to the technicality of the proof, we have decided not to do so. 
It is very likely that a number $N$ between $10$ and $100$ would work.}
$0 < \delta \leq 0.5 \cdot 10^{-4}$, and $N \geq 2.1\cdot\delta^{-1}$.
Then there exists an $\e_0>0$ such that for all $\e\leq \e_0$ and initial data $\omega_0$ with
\begin{align}
\label{initdata}
{\| \omega_0 \|}_{H^N} + {\| \omega_0 \|}_X \leq \e,
\end{align}
the equation \eqref{betaplane} admits a unique global solution $\omega \in C(\R,H^N(\R^2))$.
Moreover, for all $t \in \R$ the solution satisfies the bounds
\begin{align}
\label{maintheobounds}
 \|\omega(t)\|_{H^N}\lesssim \varepsilon_0 (1+|t|)^{C\e_0}, \qquad \|e^{-tL_1}\omega(t)\|_X \lesssim \varepsilon_0,
\end{align}
and, in particular, also the decay estimate
\begin{align}
\label{maintheodecay}
\norm{\omega(t)}_{L^\infty} \lesssim\varepsilon_0 (1+|t|)^{-1}.
\end{align}
Finally, the solutions scatters: for any initial data $\omega_0$ as in \eqref{initdata} there exist unique $f_{\pm\infty} \in X$ such that
\begin{align}
\label{maintheoscatt} 
\norm{e^{-tL_1} \omega(t) - f_{\pm\infty}}_{X} \stackrel{t \rightarrow \pm\infty}{\longrightarrow} 0.
\end{align}
\end{thm}

\vskip10pt
\subsection{Proof of the Main Theorem}\label{ssec:bootstrap}
We will prove Theorem \ref{thm:maintheo} through a bootstrap argument.
The main ingredient is the bilinear estimate \eqref{MainBound100}, which establishes Proposition \ref{prop:weight} below.
Since the equation is time reversible it suffices to consider $t>0$.
We will work with the following a priori assumptions.

\subsubsection*{A Priori Assumptions}
We assume that for some $T>0$ and $\e_1 = A \e_0$ with a suitably chosen constant $A>1$ to be determined below,
we have
\begin{align}
 \norm{P_k f(t)}_{L^2} &\leq \e_1 \ip{t}^{D \e_0} 2^{-N k^+}, \label{apriori99a}
\\
 \sup_{(k,j)\in\mathcal{J}} \big( 2^{k+j} \big)^{1+\delta} 2^{4k^+} \norm{Q_{jk}f(t)}_{L^2} &\leq \e_1, \label{apriori100a}
\end{align}
for all $t \in [0,T]$ and a suitably large $D>0$.
For small enough $T>0$ the estimates \eqref{apriori99a}-\eqref{apriori100a} hold by virtue of \eqref{initdata}
and a standard local well-posedness argument (that we omit), yielding a unique local solution such that $e^{-tL_1} \omega \in C([0,1], H^N \cap X)$.

\vskip5pt
\subsubsection*{Weighted Estimate}
As a key point in this paper we will prove:
\begin{prop}\label{prop:weight}
Assuming the a priori bounds \eqref{apriori99a}-\eqref{apriori100a}, and with the notations \eqref{FSB} and \eqref{eq:X-norm}, 
for all $t \in [0,T]$ we have
 \begin{equation}\label{MainBound100'}
\norm{ B(f,f)(t)}_{X} \lesssim \e_1^2
\end{equation}
\end{prop}
This estimate is at the heart of our article 
and its proof will be carried out over the course of the remaining Sections \ref{sec:prelim+FS}-\ref{sec:weightII}. 
In fact, we will prove the stronger version \eqref{MainBound100} of the bilinear bound \eqref{MainBound100'}, 
which also implies the scattering statement \eqref{maintheoscatt} of Theorem \ref{thm:maintheo}. 

Assuming Proposition \ref{prop:weight} we now establish the Main Theorem.

\begin{proof}[Proof of Theorem \ref{thm:maintheo}]
Our aim here is to show that the interval on which the a priori estimates \eqref{apriori99a}-\eqref{apriori100a} hold can be extended to infinity.
Using a continuity argument it will suffice to prove that for $t \in [1,T]$
\begin{equation}
\label{apriori99i}
\begin{aligned}
 \norm{P_k f(t)}_{L^2} &\leq \frac{\e_1}{2} \langle t \rangle^{D\e_0} 2^{-N k^+},
\\
 \sup_{(k,j)\in\mathcal{J}} \big( 2^{k+j} \big)^{1+\delta} 2^{4k^+} \norm{Q_{jk}f(t)}_{L^2} &\leq \frac{\e_1}{2}. 
\end{aligned}
\end{equation}

Invoking the Duhamel formula \eqref{eq:DHFT} and applying Proposition \ref{prop:weight} yields
\begin{equation*}
\begin{aligned} 
2^{4k^+} 2^{(k+j)(1+\delta)} \norm{Q_{jk}f(t)}_{L^2} &\leq 2^{4k^+} 2^{(k+j)(1+\delta)} \left(\norm{Q_{jk}\omega_0}_{L^2}+\norm{Q_{jk} B(f,f)(t)}_{L^2}\right)
\\
 &\leq \e_0+C\e_1^2\leq \frac{\e_1}{2},
\end{aligned}
\end{equation*}
for $\e_0$ small enough.
Combining this with the decay estimate \eqref{eq:dispest} we also have
\begin{align*}
 \norm{e^{tL_1} P_k f(t)}_{L^\infty} &\lesssim \ip{t}^{-1} 2^{3k} \sum_{j \geq -\min\{0,k\}} 2^j \norm{Q_{jk}f(t)}_{L^2}
 \\
 &\lesssim\ip{t}^{-1} (\e_0 + C\e_1^2)2^{-4k^+}2^{(2-\delta)k}.
\end{align*}
In particular, if $Du$ is the matrix of first derivatives of $u$, we have
\begin{align}
\label{apriori101'}
{\| \omega(t) \|}_{L^\infty} + {\| Du(t) \|}_{L^\infty} \lesssim \ip{t}^{-1} (\e_0 + C\e_1^2),
\end{align}
for all $t\in[0,T]$.
A standard energy estimate for the $\beta$-plane equation (see \cite[Lemma 3.1]{b-plane}) gives the bound
\begin{equation*}
 \norm{\omega(t)}_{H^{N}}\leq \norm{\omega(0)}_{H^{N}}\exp\left(C\int_0^t \norm{Du(s)}_{L^\infty}+\norm{\omega(s)}_{L^\infty}\; ds\right).
\end{equation*}
Inserting the decay estimate \eqref{apriori101'} and choosing appropriately the constant $D$,
it follows that
\begin{equation*}
\norm{P_k f(t)}_{L^2} \leq \e_0 \ip{t}^{D\e_0} 2^{-N k^+}.
\end{equation*}
This gives us \eqref{apriori99i} and proves the bounds \eqref{maintheobounds} and \eqref{maintheodecay} 
in our Theorem \ref{thm:maintheo}.

To conclude we remark that in proving Proposition \ref{prop:weight} we will actually prove the stronger version \eqref{MainBound100} 
of the bilinear bound \eqref{MainBound100'}.
The estimate \eqref{MainBound100} then implies that $f(t)$ is a Cauchy sequence in the $X$ space, so that \eqref{maintheoscatt} follows.
\end{proof}

\vskip10pt
\subsection{Symmetrization and Double Null Form}\label{sec:symm}
By virtue of the symmetry $\Phi(\xi,\eta) = \Phi(\xi,\xi-\eta)$ we can write the bilinear term \eqref{FSB} as
\begin{equation*}
\begin{aligned}
\mathcal{F} B(f,f)(\xi) & = \int_0^t\int_{\R^2} e^{is\Phi(\xi,\eta)}\frac{\xi\cdot\eta^\perp}{\abs{\eta}^2}\widehat{f}(\xi-\eta)\widehat{f}(\eta) d\eta ds
\\
& = \frac{1}{2} \int_0^t\int_{\R^2} e^{is\Phi(\xi,\eta)} \Big[ \frac{\xi\cdot\eta^\perp}{\abs{\eta}^2} + \frac{ \xi\cdot(\xi-\eta)^\perp}{\abs{\xi-\eta}^2} \Big]
  \widehat{f}(\xi-\eta)\widehat{f}(\eta) d\eta ds
\\
& = \frac{1}{2} \int_0^t\int_{\R^2} e^{is\Phi(\xi,\eta)} \Big[ \frac{(\xi\cdot\eta^\perp) \, \xi\cdot (\xi-2\eta)}{\abs{\eta}^2 \abs{\xi-\eta}^2} \Big]
  \widehat{f}(\xi-\eta)\widehat{f}(\eta) d\eta ds.
\end{aligned}
\end{equation*}
Here we let
\begin{align}
\label{FSBsym}
\mathfrak{m}(\xi,\eta) := \frac{1}{2}\frac{(\xi\cdot\eta^\perp) \, \xi\cdot (\xi-2\eta)}{\abs{\eta}^2 \abs{\xi-\eta}^2}
\end{align}
and explicitly write the important equality
\begin{equation}\label{FSBsymm}
\begin{aligned}
 \mathcal{F} B(f,f) &= \int_0^t\int_{\R^2} e^{is\Phi(\xi,\eta)}\frac{\xi\cdot\eta^\perp}{\abs{\eta}^2}\widehat{f}(\xi-\eta)\widehat{f}(\eta) d\eta ds
  \\
&= \int_0^t\int_{\R^2} e^{is\Phi(\xi,\eta)} \mathfrak{m}(\xi,\eta) \widehat{f}(\xi-\eta)\widehat{f}(\eta) d\eta ds.
\end{aligned}
\end{equation}

To illustrate the relevance of this symmetrization 
we remind the reader that we will treat the above expressions as oscillatory integrals. 
From this point of view, the set $\mathcal{S} = \{(\xi,\eta): \, \nabla_\eta\Phi = 0\}$ 
where no oscillations in $\eta$ occur in the phase $e^{is\Phi}$ (also called the space-resonant set)
is one of the main obstructions to obtaining strong bounds through cancellations. 
In the present problem we have
\begin{equation}\label{eq:nabla_etaPhi}
 \abs{\nabla_\eta\Phi} = \frac{\abs{\xi} \abs{\xi-2\eta}}{\abs{\xi-\eta}^2 \abs{\eta}^2},
\end{equation}
so the original multiplier $\xi\cdot\eta^\perp \abs{\eta}^{-2}$ vanishes on $\mathcal{S}$.
This is referred to as a ``null structure'' and allows one to (partially) compensate for the lack of oscillations 
(see for example \cite{K1,PusateriShatah}). However, we highlight that in our case even more is true: 
\emph{the symbol $\mathfrak{m}$ in \eqref{FSBsymm} vanishes to second order on $\mathcal{S}$}, 
which is what we call a ``double null form''. 
As we will see, this offers a crucial advantage over the previous formulation with a regular null form.

\subsubsection*{Symbol bounds}
Using the notation \eqref{defSinfty} and \eqref{mloc} we have the following basic bounds for our symbol \eqref{FSBsym}:
\begin{equation}
\label{symbound100}
\big\| \mathfrak{m}^{k,k_1,k_2} \big\| _{S^\infty} \lesssim 2^{k - \min\{k_1,k_2\}}
\end{equation}
and
\begin{equation*}
\begin{aligned}
\big\| \mathfrak{m}^{k,k_1,k_2}(\xi,\eta) \varphi_r(\eta-2\xi) \big\|_{S^\infty} \lesssim 2^{r - \min\{k_1,k_2\}},
  \\
\big\| \mathfrak{m}^{k,k_1,k_2}(\xi,\eta) \varphi_\ell(\xi-2\eta) \big\|_{S^\infty} \lesssim 2^{\ell - \min\{k_1,k_2\}},
\end{aligned}
\end{equation*}
as well as the more precise bound
\begin{equation}\label{symbound104}
\big\| \mathfrak{m}^{k,k_1,k_2}(\xi,\eta) \varphi_\ell(\xi-2\eta) \big\|_{S^\infty} \lesssim 2^{2\ell + 2k - 2k_1- 2k_2}.
\end{equation}

\vskip10pt
\subsection{Estimate for $\partial_t f$}
As a first major consequence of the symmetrization in Section \ref{sec:symm} 
we will establish a useful estimate for the time derivative of the profile.
We will work under our main a priori assumptions \eqref{apriori99a}-\eqref{apriori100a};
in order to readily have their more precise consequences \eqref{apriori101}-\eqref{aprioriH-1}
at our disposal we refer to them as they appear in \eqref{apriori99}-\eqref{apriori100}.

\begin{lemma}\label{Lemdsf}
Let $f$ be given by \eqref{defhatf}.
For all $m \in \{0,1,\dots\}$ and $t \in [2^m-1, 2^{m+1}] \cap [0,T]$, and under the a priori assumptions \eqref{apriori99}-\eqref{apriori100}, we have
\begin{align}
\label{Lemdsfconc}
\norm{P_k \partial_t f (t)}_{L^2} \lesssim  \e_1^2 2^k 2^{-4k_+} 2^{-2m + 10\delta m}.
\end{align}
\end{lemma}

Notice that $\partial_t f(t)$ is a quadratic expression in $\omega(t)$ and is therefore expected to decay, in $L^2$ at least as fast as $\norm{\omega(t)}_{L^\infty}$.
The above lemma states that we actually have much more decay, almost $t^{-2}$. This is due to the favorable ``double null structure'' of the equations.
Needless to say this estimate will be very helpful when integrating by parts in time in Duhamel's formula, which gives rise to bilinear terms involving $\partial_t f$.

\begin{proof}[Proof of Lemma \eqref{Lemdsf}]
From \eqref{defhatf} and \eqref{FSBsymm} we have
\begin{equation*}
\partial_t \widehat{f}(t) = \mathcal{F}Q(f,f)(t,\xi)
  := \frac{1}{(2\pi)^2} \int_{\R^2} e^{it\Phi(\xi,\eta)} \mathfrak{m}(\xi,\eta) \widehat{f}(t,\xi-\eta) \widehat{f}(t,\eta) d\eta.
\end{equation*}
We start by observing that for any $f,g \in L^2$ 
we have
\begin{align}
\label{LemdsfHolder}
\begin{split}
\norm{P_k Q(P_{k_1}f, P_{k_2}g)}_{L^2} \lesssim \| \mathfrak{m}^{k,k_1,k_2} \|_{S^\infty} 
  \cdot \sup_{t \approx 2^m} \min\big\{ \norm{P_{k_1} f}_{L^2} \norm{e^{itL_1} P_{k_2} g}_{L^\infty}, \\ \,
  \norm{e^{itL_1} P_{k_1} f}_{L^\infty} \norm{P_{k_2} g}_{L^2}, \, \norm{P_{k_1} f}_{L^2} \norm{P_{k_2} g}_{L^2} 2^{\min\{k_1,k_2\}} \big\},
\end{split}
\end{align}
having used Lemma \ref{lem:Holdertype}.
Moreover, notice that by symmetry in $\eta \leftrightarrow \xi-\eta$, when looking at $Q(P_{k_1}f, P_{k_2}f)$ we may assume
that $k_2 \leq k_1$ without loss of generality.

Using \eqref{LemdsfHolder} and \eqref{aprioriH-1} we see that
\begin{align*}
\begin{split}
{\| P_k Q(P_{k_1}f, P_{k_2}f) \|}_{L^2} \lesssim 2^{k-k_2} {\| P_{k_1} f \|}_{L^2} {\| P_{k_2} f \|}_{L^2} 2^{k_2}
  \lesssim 2^k \cdot \e_1 2^{-Nk_1^+} 2^{k_1} \cdot \e_1 2^{k_2},
\end{split}
\end{align*}
so that the desired conclusion follows when $k_2 \leq -2m$ or $k_1 \geq \delta m$ 
(we will choose $\delta (N - 6) \geq 2$ in \eqref{param0} below).

We also have
\begin{align*}
{\| P_k Q(P_{k_1}f, P_{k_2}f) \|}_{L^2} \lesssim 2^k {\| \mathcal{F} P_k Q(P_{k_1}f, P_{k_2}f) \|}_{L^\infty} \lesssim
  2^{2k-k_2} \cdot {\| P_{k_1} f \|}_{L^2} \cdot {\| P_{k_2} f \|}_{L^2},
\end{align*}
which, in view of \eqref{aprioriH-1}, and after summing over $k_1,k_2$ with $k_2 \geq -2m$, gives the desired bound \eqref{Lemdsfconc} if $k \leq -2m$.

In what follows we can then assume
\begin{align}
\label{Lemdsfrestr1}
\min\{k,k_1,k_2\} \geq -2m, \qquad \max\{k_1,k_2\} \leq \delta m.
\end{align}
This leaves us with a summation over $(k,k_1,k_2)$ made by at most $O(m^3)$ terms, and we see that to obtain \eqref{Lemdsfconc} it will suffice to show
\begin{align}
\label{Lemdsf10}
{\| P_k Q(P_{k_1}f, P_{k_2}f) \|}_{L^2} \lesssim \e_1^2 2^k 2^{-4k^+} 2^{-2m + 9\delta m}
\end{align}
for every fixed triple $(k,k_1,k_2)$ satisfying \eqref{Lemdsfrestr1}.
We subdivide the proof of \eqref{Lemdsf10} into two main cases: high-low and high-high interactions.

\medskip
\noindent
{\bf Case $|k_1 - k_2| \geq 10$}.
In this case we have $k_1 \geq k_2 + 10$ and $|k-k_1|\leq 5$.
We further decompose our inputs according to their spatial localization as in \eqref{FSfjkdecomp}:
\begin{align}
\label{Lemdsfjkdecomp}
 f_1 =  Q_{j_1 k_1} f, \quad f_2 = Q_{j_2 k_2} f, \quad j_\nu+k_\nu \geq 0, \, \nu=1,2.
 \end{align}
The H\"older estimate \eqref{LemdsfHolder} and the a priori bounds \eqref{apriori100}-\eqref{apriori101} give us
\begin{align*}
\begin{split}
{\| P_k Q(f_1, f_2) \|}_{L^2} \lesssim
  2^{k-k_2} \cdot \e_1 2^{-m} \cdot \e_1 2^{-k_2} 2^{-\max\{j_1,j_2\}} \cdot 2^{-2k_1^+}. 
\end{split}
\end{align*}
Therefore, we can obtain the desired bound whenever $\max\{j_1,j_2\} \geq (1-\delta)m - 2k_2$.
In the complementary case when $\max\{j_1,j_2\} \leq (1-\delta^2)m - 2k_2$ we can instead integrate by parts repeatedly in $\eta$.
More precisely, using
\begin{align*}
|\nabla_\eta \Phi| \approx 2^{-2k_2},\qquad | D_\eta^\alpha \Phi| \lesssim 2^{-(1+|\alpha|)k_2}.
\end{align*}
we can apply the bound \eqref{IBP02} in Lemma \ref{lemIBP0} with $K = s2^{-2k_2}$, $F = 2^{2k_2} \Phi$, $\epsilon = 2^{k_2}$,
and $g = \mathfrak{m}(\xi,\eta) \widehat{f_1}(\xi-\eta) \widehat{f_2}(\eta)$, and obtain
\begin{align*}
{\| P_k Q(f_1, f_2) \|}_{L^2} & \lesssim 2^k {\| \varphi_k (\xi) \what{Q}(f_1, f_2)(\xi) \|}_{L^\infty_\xi}
\\
& \lesssim 2^k \cdot (2^m 2^{-k_2})^{-M} \big( 1 + 2^{k_2} 2^{\max\{j_1,j_2\}} \big)^M \cdot 2^{k-k_2} {\| f_1 \|}_{L^2} {\| f_2 \|}_{L^2}
\\
& \lesssim \e_1^2 2^{-5m} {\| f_1 \|}_{L^2} {\| f_2 \|}_{L^2},
\end{align*}
where the last inequality follows by choosing $M$ large enough.
Using also \eqref{Lemdsfrestr1} we see that this is more than sufficient to obtain \eqref{Lemdsfconc}.

\medskip
\noindent
{\bf Case $|k_1 - k_2| < 10$}.
This case is more delicate and requires a further frequency space decomposition in the size of $|\xi-2\eta|$.
More precisely, we let
\begin{align*}
\mathcal{F} Q_\ell(f,g)(t,\xi) := \int_{\R^2} e^{it\Phi(\xi,\eta)} \mathfrak{m}(\xi,\eta) \varphi_\ell(\xi-2\eta) \widehat{f}(t,\xi-\eta) \widehat{g}(t,\eta) d\eta.
\end{align*}
Notice that this vanishes unless $\ell \leq k_1 + 20$.
To obtain \eqref{Lemdsf10} it then suffices to show
\begin{align}
\label{Lemdsf20}
\sum_{\ell \leq k_1 + 20} {\| P_k Q_\ell(P_{k_1}f, P_{k_2}f) \|}_{L^2} \lesssim \e_1^2 2^k 2^{-4k^+} 2^{-2m + 9\delta m}.
\end{align}

\smallskip
\subsubsection*{Subcase $\min\{k,\ell\} \leq (-1+5\delta)m + k_1$}
In this case we first use the $L^2\times L^\infty$ H\"older bound in Lemma \ref{lem:Holdertype}
together with the symbol bound \eqref{symbound104}, and the usual a priori estimates \eqref{apriori100}-\eqref{apriori101}, to deduce
\begin{align}
\label{Lemdsf25}
\begin{split}
2^{4k^+} {\| P_k Q_\ell (P_{k_1}f, P_{k_2}f) \|}_{L^2} \lesssim 2^{2\min\{k,\ell\}-k_1-k_2} \cdot \e_1  2^{(2-\delta)k_1} 2^{-m} \cdot \e_1 2^{k_2},
\end{split}
\end{align}
having also used \eqref{aprioriH-1}.
This suffices to obtain the desired bound when the sum in \eqref{Lemdsf20} is over $\ell \leq -m + k_1 + 5\delta m$ 
or when $k \leq -m + k_1 + 5\delta m$.

We are now left with $O(m)$ terms in the sum in \eqref{Lemdsf20}, so that it suffices to show
\begin{align}
\label{Lemdsf30}
2^{4k^+} {\| P_k Q_\ell(P_{k_1}f, P_{k_2}f) \|} \lesssim \e_1^2 2^k 2^{-2m + 8\delta m},
\end{align}
under the restrictions \eqref{Lemdsfrestr1}, $|k_1-k_2| \leq 10$ and $(-1+5\delta)m + k_1 \leq k,\ell \leq k_1 + 20$.
We now further decompose our profiles in space, letting
\begin{align*}
 P_k Q_\ell(P_{k_1}f, P_{k_2}f) = \sum_{j_1,j_2} P_k Q_\ell(f_1, f_2),
\end{align*}
with the notation \eqref{Lemdsfjkdecomp}.

\smallskip
\subsubsection*{Subcase $\max\{j_1,j_2\} \geq (1-4\delta)m - k_1 + \min\{\ell,k\}$}
In this case we use the H\"older estimate in Lemma \ref{lem:Holdertype} with the symbol bound \eqref{symbound104} to get
\begin{align*}
\begin{split}
{\| P_k Q_\ell(f_1, f_2) \|}_{L^2} \lesssim 2^{2\min\{k,\ell\} - 2k_1} \cdot
  \sup_{t \approx 2^m} \min \big\{ {\|f_1\|}_{L^2} {\|e^{itL_1} f_2\|}_{L^\infty}, {\|e^{itL_1} f_1\|}_{L^\infty} {\|f_2\|}_{L^2} \big\}.
\end{split}
\end{align*}
The a priori bounds \eqref{apriori100}-\eqref{apriori101} then give us
\begin{equation*}
\begin{aligned}
2^{4k^+} {\| P_k Q_\ell(f_1, f_2) \|}_{L^2}
  &\lesssim 2^{2\min\{k,\ell\}-2k_1} \cdot \e_1  2^{-m} 2^{(2-\delta)k_1} \cdot \e_1 2^{-k_1} 2^{-\max\{j_1,j_2\}}
\\ &\lesssim \e_1^2 2^k \cdot 2^{-\delta k_1} \cdot 2^{- m - k_1 + \min\{k,\ell\} - \max\{j_1,j_2\}},
\end{aligned}
\end{equation*}
which, upon summation over $j_1,j_2$, suffices to obtain \eqref{Lemdsf30} under the current assumptions. 

\smallskip
\subsubsection*{Subcase $\max\{j_1,j_2\} \leq (1-4\delta)m - k_1 + \min\{\ell,k\}$ and $\min\{k,\ell\} \geq (-1+5\delta)m + k_1$}
In this last remaining case we want to resort again to repeated integration by parts through Lemma \ref{lemIBP0}.

Before doing that, let us first look at the case $\ell \leq k+5$.
Notice that if $\ell \leq -m/2 + (3/2)k_1 + \delta m$, then the H\"older estimate \eqref{Lemdsf25} already gives us the desired conclusion.
We can then assume $\ell \geq -m/2 + (3/2)k_1 + \delta m$ in what follows.
On the support of $P_k Q_\ell(f_1, f_2)$ we have, see \eqref{nabla_xiPhi},
\begin{align*}
\begin{split}
\big| \nabla_\eta \Phi \big| 
  \approx 2^{\ell} 2^{-3k_1},
\qquad  \big| D_\eta^{\alpha} \Phi \big| \lesssim 2^{-(1+|\alpha|) k_1}, \quad |\alpha|\geq 2.
\end{split}
\end{align*}
We then let
\begin{align*}
K = s 2^{\ell} 2^{-3k_1}, \quad F(\eta) = \Phi(\xi,\eta) (2^{\ell} 2^{-3k_1})^{-1}
\end{align*}
and calculate
\begin{align*}
|D^{\alpha} F| \lesssim (2^{\ell} 2^{-3k_1})^{-1} 2^{-(1+|\alpha|) k_1} 
  \lesssim 2^{(1-|\alpha|)\ell}, \quad |\alpha|\geq 2.
\end{align*}
Choosing $\epsilon = 2^{\ell}$, and
$g = \mathfrak{m}(\xi,\eta) \varphi_\ell(\xi-2\eta) \widehat{f_1}(\xi-\eta) \widehat{f_2}(\eta)$, the bound \eqref{IBP02} in Lemma \ref{lemIBP0} gives us
\begin{equation*}
\begin{aligned}
 \norm{P_k Q_\ell(f_1, f_2)}_{L^2} &\lesssim  (2^m 2^{\ell} 2^{-3k_1})^{-M} \big( 2^{-\ell} + 2^{\max\{j_1,j_2\}} \big)^M
  {\| f_1 \|}_{L^2} {\| f_2 \|}_{L^2}
  \\ &\lesssim 2^{-10m} {\| f_1 \|}_{L^2} {\| f_2 \|}_{L^2},
\end{aligned}
\end{equation*}
which is more than enough.

Finally we look at the case $k \leq \ell-5$. Recall that we may assume $k \geq -m + k_1 + 5\delta m$.
In the present configuration we have
\begin{align*}
\begin{split}
\big| \nabla_\eta \Phi \big| \approx 2^{k} 2^{-3k_1},
\qquad  \big| D_\eta^{\alpha} \Phi \big| \lesssim 2^{-(2+|\alpha|) k_1} 2^k, \quad |\alpha|\geq 2.
\end{split}
\end{align*}
We can then apply Lemma \ref{lemIBP0} with
$K = s 2^{k} 2^{-3k_1}, F(\eta) = \Phi(\xi,\eta) (2^{k} 2^{-3k_1})^{-1}$, $\epsilon = 2^{k_1}$, and the same choice of $g$ as above,
to obtain ${\| P_k Q_\ell(f_1, f_2) \|}_{L^2}\lesssim 2^{-5m} {\| f_1 \|}_{L^2} {\| f_2 \|}_{L^2}$.
This concludes the proof of the lemma.
\end{proof}

\vskip20pt
\section{Preliminary Bounds and Finite Speed of Propagation}\label{sec:prelim+FS}
Recall that our aim is to prove Proposition \ref{prop:weight}.
We begin by localizing our time parameter on scales $\approx 2^m$, $m \in \N$ as follows.
Given $t\in[0,T]$, we choose a suitable decomposition of the indicator function $\mathbf{1}_{[0,t]}$
by fixing functions $\tau_0,\ldots,\tau_{L+1} : \R \to [0,1]$, $\abs{L-\log_2(2+t)}\leq 2$ with the properties
\begin{align*}
\begin{split}
& \mathrm{supp} \,\tau_0 \subseteq [0,2], \quad \mathrm{supp} \,\tau_{L+1}\subseteq [t-2,t],
  \quad \mathrm{supp}\,\tau_m\subseteq [2^{m-1},2^{m+1}]
  \quad \text{for} \quad  m \in \{1,\dots,L\},
\\
& \sum_{m=0}^{L+1}\tau_m(s) = \mathbf{1}_{[0,t]}(s), \quad \tau_m\in C^1(\R) \quad \text{and} \quad \int_0^t|\tau'_m(s)|\,ds\lesssim 1
  \quad \text{for} \quad m\in \{1,\ldots,L\}.
\end{split}
\end{align*}
We can then decompose
\begin{align*}
\begin{split}
B(f,f) = \sum_{m} B_m(f,f),
\qquad \mathcal{F} B_m(f,f):= \int_0^t \tau_m(s) \int_{\R^2} e^{is\Phi(\xi,\eta)} \mathfrak{m}(\xi,\eta) \widehat{f}(\xi-\eta)\widehat{f}(\eta) d\eta ds.
\end{split}
\end{align*}
To obtain Proposition \ref{prop:weight} it will then suffice to show that for any $m =0,1,\dots$
\begin{equation}
\label{MainBound100}
 2^{4k^+} 2^{(k+j)(1+\delta)} \norm{Q_{jk} B_m(f,f)}_{L^2} \lesssim \e_1^2 2^{-\delta^3 m}.
\end{equation}

For convenience we recall here the a priori bounds \eqref{apriori99a}-\eqref{apriori100a},
\begin{align}
\label{apriori99}
\norm{P_k f(t)}_{L^2} &\leq \e_1 \ip{t}^{p_0} 2^{-N_0 k^+}, 
\\
\label{apriori100}
\sup_{(k,j)\in\mathcal{J}} \big( 2^{k+j} \big)^{1+\delta} 2^{4k^+} \norm{Q_{jk}f(t)}_{L^2} &\leq \e_1,
\end{align}
where we can choose $p_0 = C\e_0 \leq \delta$, for a suitable absolute constant $C>0$.
Then we also have the following consequences of \eqref{apriori99}-\eqref{apriori100}:
\begin{align}
\label{apriori101}
 &\norm{e^{itL_1} Q_{jk}f(t)}_{L^\infty}
 \lesssim \e_1\ip{t}^{-1}2^{-4k^+}2^{(2-\delta)k}2^{-\delta j},
\\
\label{apriori102}
 & {\| \widehat{Q_{jk}f} \|}_{L^\infty} \leq \norm{Q_{jk}f}_{L^1} \lesssim \e_1 2^{-(1+\delta) k} 2^{-4k^+} 2^{-\delta j}.
\end{align}
Also recall that by virtue of \eqref{eq:L2cons'} we have
\begin{equation}\label{aprioriH-1}
 2^{-k}\norm{P_k f}_{L^2}\lesssim \big\| \abs{\nabla}^{-1} f \big\|_{L^2}= \big\|\abs{\nabla}^{-1} \omega\big\|_{L^2} \lesssim \norm{u}_{L^2} \leq \e_0.
\end{equation}

In the remainder of this section we begin our proof of the weighted estimate \eqref{MainBound100} 
by treating first some ranges of parameters for which the estimates are easily seen to hold. 
Subsequently we present a ``finite speed of propagation'' argument, which invokes the idea that 
each frequency is expected to travel at its respective group velocity, in order to allow for a further reduction in the parameters to be considered.

\subsection{Basic Cases}
We first establish a simple lemma dealing with frequencies that are very large or very small with respect to the relevant parameters.
To this end we let
\begin{equation}\label{param0}
 N^\prime:= N-6, \quad N^\prime \geq \frac{2}{\delta}.
\end{equation}

\begin{lemma}[Basic Cases]
With the above notation and under the a priori assumptions \eqref{apriori99}-\eqref{apriori101} we have
\begin{align}
\label{Basiccases1}
\sum_{\max\{k_1,k_2\} \geq (k + j + \delta m)/N^\prime}
  2^{4k^+} 2^{(1+\delta)(k+j)} {\| Q_{jk} B_m( P_{k_1}f, P_{k_2}f) \|}_{L^2} \lesssim 2^{-\delta^3 m} \e_1^2.
\end{align}
Moreover,
\begin{align}
\label{Basiccases2}
\sum_{\min\{k_1,k_2\} \leq - 1.01(k + j + \delta m)}
  2^{4k^+} 2^{(1+\delta)(k+j)} {\| Q_{jk} B_m( P_{k_1}f, P_{k_2}f) \|}_{L^2} \lesssim 2^{-\delta^3 m} \e_1^2.
\end{align}
\end{lemma}

\begin{proof}
We begin by using an $L^2 \times L^\infty$ estimate, see Lemma \ref{lem:Holdertype},
together with the symbol bound \eqref{symbound100}, to deduce that
\begin{equation}
\label{Bascas1}
\begin{aligned}
 \norm{Q_{jk} B_m( P_{k_1}f, P_{k_2}f)}_{L^2} &\lesssim  2^m \cdot 2^{k-\min\{k_1,k_2\}}
  \cdot \sup_{t \approx 2^m} \min \big\{ {\| P_{k_1} f \|}_{L^2} {\| e^{itL_1} P_{k_2} f \|}_{L^\infty},
  \\
  &\qquad {\| e^{itL_1} P_{k_1} f \|}_{L^\infty} {\| P_{k_2} f \|}_{L^2}, \, {\| P_{k_1} f \|}_{L^2} {\| P_{k_2} f \|}_{L^2} 2^{\min\{k_1,k_2\}} \big\}.
\end{aligned}
\end{equation}

\medskip
\noindent
{\it Proof of \eqref{Basiccases1}.}
Without loss of generality, let us assume $k_2 \leq k_1$, so that the sum is over $k_1 \geq (k + j + \delta m)/N^\prime$.
Using the bound in the high Sobolev norm \eqref{apriori99}, the a priori decay assumption \eqref{apriori101},
and the estimate \eqref{Bascas1} above, we see that
\begin{align*}
{\| Q_{jk} B_m( P_{k_1}f, P_{k_2}f) \|}_{L^2} \lesssim  2^m \cdot 2^{k-k_2} \cdot \e_1 2^{-m} 2^{(2-\delta)k_2} 2^{-4k_2^+}
  \cdot \e_1 2^{p_0 m} 2^{-N k_1}.
\end{align*}
It follows that
\begin{align*}
\sum_{k_1 \geq (k + j + \delta m)/N^\prime, k_2}\hspace{-5pt} 2^{4k^+} 2^{(1+\delta)(k+j)}
  {\| Q_{jk} B_m( P_{k_1}f, P_{k_2}f) \|}_{L^2} \lesssim 2^{(1+\delta)(k+j)} \cdot \e_1^2 2^{p_0 m} 2^{-(N-5) (k + j + \delta m)/N^\prime}.
\end{align*}
Since $(N-5)/N^\prime \geq 1+\delta$ and $p_0 \leq \delta$ this is sufficient.

\medskip
\noindent
{\it Proof of \eqref{Basiccases2}.}
Again, without loss of generality we assume $k_2 \leq k_1$, so that the sum is over $k_2 \leq -1.01(k + j + \delta m)$.
Using the estimate \eqref{Bascas1} above, the a priori bounds \eqref{apriori100}, \eqref{apriori101} and \eqref{aprioriH-1}, we see that
\begin{align*}
{\| Q_{jk} B_m( P_{k_1}f, P_{k_2}f) \|}_{L^2} \lesssim  2^m \cdot 2^{k-k_2} \cdot \e_1 2^{-m} 2^{(2-\delta)k_2}
  \cdot \e_1 2^{k_1} 2^{-4 k_1^+}.
\end{align*}
It follows that
\begin{align*}
\sum_{k_2 \leq -1.01(k + j + \delta m), k_1} 2^{4k^+} 2^{(1+\delta)(k+j)}
  {\| Q_{jk} B_m( P_{k_1}f, P_{k_2}f) \|}_{L^2} \lesssim 2^{(1+\delta)(k+j)} \cdot \e_1^2 2^{-(1-\delta)1.01(k + j + \delta m)}
\end{align*}
which is sufficient for $\delta \leq 1/1000$.
\end{proof}

As a consequence of the above lemma we can assume from now on that
\begin{align}
\label{Basiccasescons}
\max\{k_1,k_2\} 
\leq \frac{\delta}{2}(k + j + \delta m), \qquad \min\{k_1,k_2\} \geq - 1.01(k + j + \delta m)
\end{align}
and, in particular,
\begin{align}
\label{Basiccasescons'}
\max\{k,k_1,k_2\}\leq \delta (j+\delta m) + D.
\end{align}
where $D$ is a suitably large constant.
From now on we will use $D$ to denote an absolute constant that needs to be chosen large enough 
in the course of our proof so to verify several inequalities.
In view of \eqref{Basiccasescons}-\eqref{Basiccasescons'}, 
when decomposing our inputs into frequencies, summations are given by at most $O((j+m)^2)$ terms.

\medskip
\subsection{Finite Speed of Propagation}
From \eqref{defPhi} one computes
\begin{align}\label{nabla_xiPhi}
 \abs{\nabla_\xi\Phi} = \frac{\abs{\eta} \abs{\eta-2\xi}}{\abs{\xi-\eta}^2 \abs{\xi}^2},
\qquad \abs{\nabla_\eta\Phi} = \frac{\abs{\xi} \abs{\xi-2\eta}}{\abs{\xi-\eta}^2 \abs{\eta}^2}.
\end{align}
Notice that applying a weight $x$ to the bilinear term $B(f,f)$ corresponds to differentiating in $\xi$ its Fourier transform, i.e.
the expression in \eqref{FSB}.
The main contribution from this can be expected to be the term where the $\xi$-derivative hits the oscillating phase, producing a
factor of $s \nabla_\xi \Phi$.
We then want to make this statement precise by proving that if the bilinear term $B(f,f)$ is restricted to locations $|x| \approx 2^j$,
then we must have ``$\abs{x} \lesssim s \abs{\nabla_\xi\Phi}$'', 
that is, we should expect to have $2^j \lesssim 2^m 2^{-2\min\{k,k_2,k_2\}}$.
Later on in Section \ref{sec:weightI} we will also use refinements of this statement in various scenarios.

\medskip
\begin{lemma}[Finite speed of propagation]\label{FSlemma}
Assume that \eqref{Basiccasescons'} holds and
\begin{align}
\label{FSassumption}
j \geq m - 2\min\{k,k_1,k_2\} + D^2,
\end{align}
then we have the bound
\begin{align}
\label{FSconclusion}
2^{4k^+} 2^{(k+j)(1+\delta)} \norm{Q_{jk} B_m (P_{k_1}f, P_{k_2}f)}_{L^2} \lesssim 2^{-\delta^2 (m+j)} \e_1^2.
\end{align}
\end{lemma}

\medskip
\begin{proof}
We subdivide the proof in several cases and subcases.

\medskip
\noindent
{\bf Case 1: $k_1 \geq k_2+10$}.
In this case we must have $|k_1 - k| \leq 10$ and the assumption \eqref{FSassumption} implies
\begin{align}
\label{FSassumption1}
j \geq m - 2k_2 + D^2.
\end{align}
Notice that in view of \eqref{Basiccasescons'} we must have $j \geq m/2$.

\medskip
\noindent
{\it Subcase 1.1: $k \leq -(1-\delta^2)j$}.
In this case 
we can use an $L^2\times L^\infty$ estimate, see Lemma \ref{lem:Holdertype} and the symbol bound \eqref{symbound100},
with the a priori bounds \eqref{apriori100}-\eqref{apriori101}, to obtain
\begin{equation*}
\begin{aligned}
&2^{(1+\delta)(k+j)} \norm{Q_{jk} B_{m}(P_{k_1}f, P_{k_2}f)}_{L^2}
  \lesssim 2^{\delta j} \norm{Q_{jk} B_{m}(P_{k_1}f, P_{k_2}f)}_{L^2}
\\
&\quad \lesssim 2^{\delta j} \cdot 2^m \cdot 2^{k - k_2} \cdot \sup_{t \approx 2^m} {\| P_{k_1} f \|}_{L^2} {\| e^{itL_1} P_{k_2} f \|}_{L^\infty}
\\
&\quad \lesssim 2^{\delta j} \cdot 2^m \cdot 2^{k} \cdot \e_1 \cdot \e_1 2^{-m},
\end{aligned}
\end{equation*}
which suffices to obtain \eqref{FSconclusion}.

We now further decompose the profiles according to their spatial localization by defining, see \eqref{phijk}-\eqref{Qjkdecomp},
\begin{align}
\label{FSfjkdecomp}
 f_1 =  Q_{j_1 k_1} f, \quad f_2 = Q_{j_2 k_2} f, \quad j_\nu+k_\nu \geq 0, \, \nu=1,2. 
\end{align}

\medskip
\noindent
{\it Subcase 1.2: $\min\{j_1,j_2\} \geq (1-\delta^2)j$}.
Here we use again an $L^2 \times L^\infty$ estimate and the a priori bounds \eqref{apriori100}-\eqref{apriori101}:
\begin{align*}
& 2^{4k^+} 2^{(1+\delta)(k+j)} \norm{Q_{jk} B_{m}(f_1, f_2)}_{L^2}
\\
&\quad \lesssim 2^{4k^+} 2^{(1+\delta)(k+j)} \cdot 2^m \cdot 2^{k - k_2} \cdot \sup_{t \approx 2^m} {\| f_1 \|}_{L^2} {\| e^{itL_1} f_2 \|}_{L^\infty}
\\
&\quad \lesssim 2^{5k^+} 2^{(1+\delta)(k+j)} \cdot 2^m \cdot \e_1 2^{-4k_1^+} 2^{-(1+\delta) (k_1+j_1)} \cdot \e_1 2^{-m} 2^{-4k_2^+} 2^{(1-\delta)k_2} 2^{-\delta j_2}.
\end{align*}
Using the assumption $\min\{j_1,j_2\} \geq (1-\delta^2)j$ this can be bounded by
\begin{align*}
\e_1^2 2^{k^+} 2^{(1+\delta) j} \cdot  2^{-(1+\delta) j_1} \cdot 2^{-\delta j_2}
  \lesssim \e_1^2 2^{k^+}  2^{-4\delta j/5} 2^{-\delta^2 j_1} 2^{-\delta^2 j_2}.
\end{align*}
Upon summing over $j_1$ and $j_2$ we obtain the bound \eqref{FSconclusion} also in view of $k \leq 2\delta j/3 + \delta^2 m+D$,
see \eqref{Basiccasescons}. 

\medskip
\noindent
{\it Subcase 1.3: $-k, \min\{j_1,j_2\} \leq (1-\delta^2)j$}.
In this case we want to integrate by parts in $\xi$ using the main assumption \eqref{FSassumption}.
More precisely, let us decompose according to \eqref{FSfjkdecomp} and inspect the formula
\begin{align}
 \label{FSIBPxi0}
\begin{split} 
\varphi_j^{(k)}(x) P_k B_{m}(f_1,f_2)(x)
= \varphi_j^{(k)}(x) \int_0^t \tau_m(s) \int_{\R^2\times\R^2}
  e^{i[ x\cdot\xi + s\Phi(\xi,\eta) ]} \mathfrak{m}(\xi,\eta) \varphi_k(\xi)
  \\ \times \widehat{f_1}(\xi-\eta) \widehat{f_2}(\eta) \, d\eta d\xi \, ds.
\end{split}
\end{align}
Let us assume first that $j_1 \leq (1-\delta^2)j$.
Notice that \eqref{nabla_xiPhi} and the hypothesis \eqref{FSassumption1} imply
\begin{align}
 \label{FSIBPxi0'}
\big| \nabla_{\xi} \big[ x\cdot\xi + s\Phi(\xi,\eta) \big] \big| = \big| x + s \nabla_\xi \Phi \big| \gtrsim 2^j.
\end{align}
We then want to apply Lemma \ref{lemIBP0} to
\begin{align*}
\int_{\R^2} e^{i[ x\cdot\xi + s\Phi(\xi,\eta) ]} \mathfrak{m}(\xi,\eta) \varphi_k(\xi) \widehat{f_1}(\xi-\eta) \, d\xi.
\end{align*}
Let us explain this in detail since similar arguments will be used repeatedly below.
We let
\begin{align}
\label{FSIBPphase}
F(\xi) & =  2^{-j} \big[ x\cdot\xi + s\Phi(\xi,\eta) \big], \quad K \approx 2^j,
\end{align}
and have, for $\abs{\alpha} \geq 2$,
\begin{align*}
 \abs{D^{\alpha} F} \lesssim  2^{-j} s \, \abs{D^{\alpha}_\xi \Phi(\xi,\eta)} \lesssim 2^{-j + m} 2^{-(\abs{\alpha}+1)\min\{k,k_1\}} \lesssim 2^{(1-\abs{\alpha})\min\{k,k_1\}}.
\end{align*}
We can then choose $\epsilon = 2^{\min\{k,k_1\}}$,
make the natural choice of the integrand $$g(\xi) = \mathfrak{m}(\xi,\eta) \varphi_k(\xi) \widehat{f_1}(\xi-\eta),$$
and use the bound \eqref{IBP02} to obtain
\begin{align*}
{\| Q_{jk} B_{m}(f_1,f_2) \|}_{L^2} \lesssim
  2^{m+j} \cdot {\| \what{f_2} \|}_{L^1} \cdot \frac{1}{(2^{j+\min\{k,k_1\}})^M} \sum_{|\alpha|\leq M} 2^{\min\{k,k_1\} |\alpha|} \, {\| D^{\alpha} g \|}_{L^1}
\\
\lesssim 2^{m+j} \e_1  \cdot 2^{-jM} \big[ 2^{-\min\{k,k_1\}M} + 2^{-k M} + 2^{j_1 M}]\cdot \e_1
\lesssim 2^{-10j} \e_1^2.
\end{align*}
For the last inequality we have used \eqref{symbound100}, 
the fact that $\max\{-k,-k_1,j_1\} \leq (1-\delta^2)j$, and chosen $M = O(\delta^{-2})$ sufficiently large.
This gives \eqref{FSconclusion} when $j_1 \leq (1-\delta^2)j$.

When $j_2 \leq (1-\delta^2)j$ we can use a similar argument. More precisely we look at the formula \eqref{FSIBPxi0} and change variables to write
\begin{align*}
\begin{split}
Q_{jk} B_{m}(f_1,f_2)(x)
= \varphi_j^{(k)}(x) \int_0^t \tau_m(s) \int_{\R^2\times\R^2}
  \left[ e^{i[ x\cdot\xi + s\Phi(\xi,\eta) ]} \varphi_k(\xi) \mathfrak{m}(\xi,\xi-\eta) \right.
  \\ \left. \times \widehat{f_2}(\xi-\eta) \, d\xi \right] \widehat{f_1}(\eta) d\eta \, ds.
\end{split}
\end{align*}
Notice that \eqref{FSIBPxi0'} still holds. Therefore we can apply Lemma \ref{lemIBP0} 
with the same phase as in \eqref{FSIBPphase} above, $\epsilon = 2^{-k_2}$, and the natural choice of the integrand $g$, obtaining
\begin{align*}
{\| Q_{jk} B_{m}(f_1,f_2) \|}_{L^2}
\lesssim 2^{m+j} \cdot 2^{ -(j+k_2)M} \e_1^2 \big[ 1 + 2^{(k_2+j_2)M} \big] \lesssim 2^{-10j} \e_1^2,
\end{align*}
since $-k_2 \leq j_2 \leq (1-\delta^2)j$.

\medskip
\noindent
{\bf Case 2: $k_2 \geq k_1+10$}.
This case is completely analogous to Case 1 since our main assumption is symmetric upon exchanging $k_1$ and $k_2$.

\medskip
\noindent
{\bf Case 3: $|k_1 - k_2| \leq 10$}.
In this case we have
\begin{equation*}
k \leq \min\{k_1,k_2\} + 20,
\end{equation*}
and the main assumption \eqref{FSassumption} implies
\begin{equation*}
j \geq m - 2k + D.
\end{equation*}
Recall that in view of \eqref{Basiccasescons'} we must have $j \geq m/2$.
Also, using the same estimate of Subcase 1.1 above, we may assume $k \geq -(1-\delta^2)j$.

\smallskip
\noindent
{\it Subcase 3.1: $\min\{j_1,j_2\} \geq (1-\delta^2)j$}.
This case can be treated like we have done in the analogous subcases above via an $L^\infty \times L^2$ estimate: 
\begin{align*}
& 2^{4k^+} 2^{(1+\delta)(k+j)} \norm{Q_{jk} B_{m}(f_1,f_2)}_{L^2}
\\
&\quad \lesssim 2^{4k^+} 2^{(1+\delta)(k+j)} \cdot 2^m \cdot \sup_{t \approx 2^m} {\| e^{itL_1} P_{k_1} f \|}_{L^\infty} {\| P_{k_2} f \|}_{L^2}
\\
&\quad \lesssim \e_1^2 2^{(1+\delta) j} \cdot  2^{-\delta j_1} \cdot  2^{-(1+\delta) j_2} \lesssim \e_1^2 2^{-\delta j/2} 2^{-\delta^2 j_1} 2^{-\delta^2 j_2}.
\end{align*}
Summing over $j_1,j_2$ we get the desired bound \eqref{FSconclusion}.

\smallskip
\noindent
{\it Subcase 3.2: $\min\{j_1,j_2\} \leq (1-\delta^2)j$}.
In this case we can integrate by parts in $\xi$ as previously done after \eqref{FSIBPxi0}, using Lemma \ref{lemIBP0},
the lower bound \eqref{FSIBPxi0'} and $-k \leq (1-\delta^2)j$.
\end{proof}

\vskip20pt
\section{The Weighted Estimate: Part I}\label{sec:weightI}
In this section we begin the proof of the main weighted bound
\begin{equation}
\label{eq:w_bound}
 \sup_{(k,j)\in\mathcal{J}}2^{4k^+}2^{(k+j)(1+\delta)}\norm{Q_{jk}B_{m}(f,f)}_{L^2}\lesssim 2^{-\delta^3 m}\varepsilon_1^2,
\end{equation}
showing how this can be reduced to a similar one
where the size of various important quantities can be restricted to specific ranges depending on the time variable.
More precisely we will show how to restrict the size of the input and output
frequencies to a range close to $1$ (a range of the form $[2^{-c_1\delta m}, 2^{c_2\delta m}]$ for some constants $c_1,c_2>0$),
the size of the phase $\Phi=\Phi(\xi,\eta)$ close to $2^{-m}$, and the size of its gradients in $\xi$ and $\eta$ close to $1$.
In Section \ref{sec:weightII} we will then conclude our proof by treating the remaining cases.

\medskip
\vskip10pt
\subsection{Main Reduction of Interaction Frequencies}
Here we show how to treat the contributions from input and output frequencies that are much smaller than $1$,
more precisely smaller than $2^{-c\delta m}$ for some $c>0$.

\begin{prop}\label{prop:red_deltas}
Under the a priori assumptions \eqref{apriori100}-\eqref{apriori101} we have, for all $(k,j) \in \mathcal{J}$,
\begin{align}
\label{eq:weightbound}
\begin{split}
&\sum_{\substack{|k_1-k_2| \geq 10 \\ \min\{k_1,k_2\} \leq -5\delta m + D}}
   2^{4k^+}2^{(k+j)(1+\delta)} \| Q_{jk}B_{m}(P_{k_1}f,P_{k_2}f) \|_{L^2} \lesssim 2^{-2\delta^3 m} \e_1^2.
\end{split}
\end{align}
Furthermore, for all $(k,j) \in \mathcal{J}$ 
we have
\begin{equation}
\label{eq:weightbound'}
  \sum_{|k_1-k_2| \leq 10} 2^{4k^+}2^{(k+j)(1+\delta)} 
  {\| Q_{jk}B_{m}(P_{k_1}f,P_{k_2}f) \|}_{L^2} \lesssim 2^{-2\delta^3 m} \e_1^2 \quad \text{ if }\; k \leq -5\delta m + D.
\end{equation}
\end{prop}

\begin{proof}[Proof of Proposition \ref{prop:red_deltas}]
We split the proof into several scenarios, the most difficult ones being 
the high-high interactions.

\bigskip
\noindent
{\it Proof of \eqref{eq:weightbound}.} 
Because of the symmetry in $k_1,k_2$ we may assume $k_2+10\leq k_1$, $|k-k_1| \leq 10$.

\medskip
\noindent
{\it Case 1: $k \leq -(1-\delta^2)j$}.
In this case we can use an $L^2\times L^\infty$ estimate, see Lemma \ref{lem:Holdertype} and the symbol bound \eqref{symbound100},
with the a priori bounds \eqref{apriori100}-\eqref{apriori101} to obtain
\begin{equation*}
\begin{aligned}
 2^{(1+\delta)(k+j)} \norm{Q_{jk} B_{m}(P_{k_1}f, P_{k_2}f)}_{L^2}
&\lesssim 2^{\delta j} \cdot 2^m \cdot 2^{k - k_2} \cdot \sup_{t \approx 2^m} \norm{P_{k_1} f}_{L^2} \norm{e^{itL_1} P_{k_2} f}_{L^\infty}
\\
&\lesssim 2^{\delta j} \cdot 2^m \cdot 2^{k} \cdot \e_1 \cdot \e_1 2^{-m} 2^{(1-\delta) k_2},
\end{aligned}
\end{equation*}
which suffices to obtain \eqref{eq:weightbound}. From now on we may assume $-k \leq (1-\delta^2)j$.

Let us now decompose the profiles according to their spatial localization, adopting the same notation as in \eqref{FSfjkdecomp}:
\begin{align}
\label{prored20}
 f_1 =  Q_{j_1 k_1} f, \quad f_2 = Q_{j_2 k_2} f, \quad j_\nu+k_\nu \geq 0, \, \nu=1,2. 
\end{align}

\medskip
\noindent
{\it Case 2:  $j_1 \geq (1-\delta^2)j$}.
Here we use again an $L^2 \times L^\infty$ estimate and the a priori bounds \eqref{apriori100}-\eqref{apriori101}:
\begin{align*}
& 2^{4k^+} 2^{(1+\delta)(k+j)} \norm{Q_{jk} B_m(f_1, f_2)}_{L^2}
\\
&\qquad \lesssim 2^{4k^+} 2^{(1+\delta)(k+j)} \cdot 2^m \cdot 2^{k - k_2} \cdot \sup_{t \approx 2^m} {\| f_1 \|}_{L^2} {\| e^{itL_1} f_2 \|}_{L^\infty}
\\
&\qquad \lesssim 2^{5k^+} 2^{(1+\delta)(k+j)} \cdot 2^m \cdot \e_1 2^{-4k_1^+} 2^{-(1+\delta) (k_1+j_1)} \cdot \e_1 2^{-m} 2^{(1-\delta)k_2} 2^{-\delta j_2}.
\end{align*}
Using the assumption $j_1 \geq (1-\delta^2)j$, the finite speed of propagation Lemma \ref{FSlemma} to bound $j \leq m - 2k_2 + D$,
and that $k \leq 4\delta j/5+ \delta^2 m+D$ by \eqref{Basiccasescons}, we can estimate
\begin{align*}
2^{4k^+} 2^{(1+\delta)(k+j)} \norm{Q_{jk} B_m(f_1, f_2)}_{L^2}
  &\lesssim \e_1^2 2^{k^+} 2^{3 \delta^2 j} \cdot  2^{-\delta^2 j_1} \cdot 2^{(1-\delta)k_2} 2^{-\delta j_2}
\\
  &\lesssim \e_1^2 2^{2\delta m} \cdot 2^{(1-3\delta)k_2} 2^{-\delta^2 j_1} 2^{-\delta j_2}.
\end{align*}
Summing over $j_1$ and $j_2$ we obtain \eqref{eq:weightbound}. From now on we may assume $j_1 \leq (1-\delta^2)j$.

\medskip
\noindent
{\it Case 3: $j \geq k_2 - 3k_1 + m + D$}.
In this case we proceed in a similar way as we did in the proof of Lemma \ref{FSlemma}, resorting to integration by parts in $\xi$.
We look again at the formula \eqref{FSIBPxi0} and notice that $|\nabla_\xi \Phi| \approx 2^{k_2 - 3k_1}$, see \eqref{nabla_xiPhi}.
Then we have the same lower bound as in \eqref{FSIBPxi0'}, that is
\begin{align*}
\big| \nabla_{\xi} \big[ x\cdot\xi + s\Phi(\xi,\eta) \big] \big|  \gtrsim 2^j,
\end{align*}
and we can apply Lemma \ref{lemIBP0} to
\begin{align*}
\int_{\R^2} e^{i[ x\cdot\xi + s\Phi(\xi,\eta) ]} \mathfrak{m}(\xi,\eta) \varphi_k(\xi) \widehat{f_1}(\xi-\eta) \, d\xi.
\end{align*}
More precisely we do this by choosing again $F(\xi) =  2^{-j} [ x\cdot\xi + s\Phi(\xi,\eta) ]$, $K=2^j$, and using that for $|\alpha| \geq 2$
\begin{align*}
| D^{\alpha} F | \lesssim  2^{-j} s \, | D^{\alpha}_\xi \Phi(\xi,\eta) | \lesssim 2^{-j + m} \cdot 2^{-(|\alpha|+2)\min\{k,k_1\}} 2^{k_2}
  \lesssim 2^{(1-|\alpha|)k_1},
\end{align*}
so that we can let $\epsilon = 2^{k_1}$.
Using the bound \eqref{IBP02}, and the a priori bounds \eqref{apriori100} and \eqref{aprioriH-1}, we can deduce
\begin{align*}
{\| Q_{jk} B_{m}(f_1,f_2) \|}_{L^2} \lesssim 2^m 2^{-10j} \cdot 2^{k_1-k_2} \cdot {\| \what{f_1} \|}_{L^1} {\| \what{f_2} \|}_{L^1}
\lesssim 2^{-5j} 2^{-2k_1^+} \e_1^2,
\end{align*}
which can be multiplied by the factor $2^{(j+k)(1+\delta)}$ and summed over all indices 
to give the desired estimate. 
From now on we may assume $j \leq k_2 - 3k_1 + m + D$.

\medskip
\noindent
{\it Case 4}: $\max\{j_1,j_2\} \geq m-2k_2-\delta^2m$.
We use an H\"older estimate together with the usual a priori bounds,
placing the term with larger localization in $L^2$ and the other one in $L^\infty$, and obtain 
\begin{align*}
& 
2^{4 k^+} 2^{(k_2 - 2k_1 + m)(1+\delta)} \| P_k B_{m}(f_1,f_2)\|_{L^2}
\\
& \qquad \lesssim 2^{(k_2-2k_1+m)(1+\delta)} \cdot 2^m 2^{k_1-k_2}
  \cdot 2^{-m} 2^{(2-\delta)k_1}\e_1 \cdot 2^{-\max\{j_1,j_2\}} 2^{-(1+\delta)k_2}\e_1 \cdot 2^{-\delta (j_1+j_2)}
\\
& \qquad \lesssim \varepsilon_1^2 2^{2\delta m} 2^{(1-3\delta)k_1} 2^{k_2} 2^{-\delta(j_1+j_2)}.
\end{align*}
Also in view of $j \leq - 2k_1 + m + D$ and \eqref{Basiccasescons'} we have $k_1\leq 2\delta m + D$,
thus summing the bound above over $j_1,j_2$ we obtain \eqref{eq:weightbound} whenever $k_2 \leq -5\delta m$.

\medskip
\noindent
{\it Case 5}: $\max\{j_1,j_2\} \leq m-2k_2-\delta^2m$.
Notice that since $k_2 \leq k_1 -10$ we have, see \eqref{nabla_xiPhi},
\begin{align*}
|\nabla_\eta \Phi(\xi,\eta)| \approx 2^{-2k_2}, \quad |D_\eta^{\alpha} \Phi(\xi,\eta)| \lesssim 2^{-k_2(|\alpha|-1)}, \quad |\alpha|\geq 2.
\end{align*}
We then resort to multiple integrations by parts in $\eta$, that is, we apply Lemma \ref{lemIBP0} with $F = 2^{2k_2} \Phi$, $K = s 2^{-2k_2}$,
$\epsilon = 2^{k_2}$ and $g = \mathfrak{m}(\xi,\eta)\widehat{f_1}(\xi-\eta) \widehat{f_2}(\eta)$.
Using the bound \eqref{IBP02} we have
\begin{align*}
\| Q_{jk}B_{m}(f_1,f_2) \|_{L^2} \lesssim 2^k \big\| \mathcal{F}(Q_{jk}B_{m}(f_1,f_2)) \big\|_{L^\infty}
  \lesssim 2^k 2^{-10m} 2^{k_1-k_2} {\| f_1 \|}_{L^2} {\| f_2 \|}_{L^2},
\end{align*}
which is more than sufficient to obtain \eqref{eq:weightbound} using also $j+k \leq k_2 -2k_1 + m + D$ and \eqref{apriori100}-\eqref{aprioriH-1}.

\bigskip
\noindent
{\it Proof of \eqref{eq:weightbound'}}.
In this scenario we will make crucial use of the symmetrization argument which gives better bounds on the null structure.
In view of Lemma \ref{FSlemma} (and the assumption that $k\leq -5\delta m + D$), in the current frequency configuration it is enough to show
\begin{align}
\label{prored100'}
\sum_{|k_1-k_2| \leq 10} 2^{(-k + m)(1+\delta)} {\| P_k B_{m}(P_{k_1}f,P_{k_2}f) \|}_{L^2} \lesssim \e_1^2 2^{-2\delta^3 m}.
\end{align}

\medskip
\noindent
{\it Localization in the size of $|\xi-2\eta|$}.
We now introduce a further localization in the size of $|\xi-2\eta|$ by writing
\begin{align}
\label{prored101}
\begin{split}
& \mathcal{F}B_{m,\ell}(f,g) = \int_0^t \tau_m(s)\int_{\R^2} W_{\ell}(f,g)\, d\eta ds,
\\
& W_{\ell}(f,g) := e^{is\Phi} \mathfrak{m}(\xi,\eta) \varphi_\ell(\xi-2\eta) \what{f}(\xi-\eta) \what{g}(\eta).
\end{split}
\end{align}
Notice that $B_{m,\ell}(P_{k_1}f, P_{k_2}f)$ vanishes if $\ell \geq k_1 + 20$.
Also, notice that the symbol obeys the refined bound
\begin{align}
\label{proredsymbound}
{\| \mathfrak{m}^{k,k_1,k_2} \varphi_\ell(\xi-2\eta) \|}_{S^\infty} \lesssim 2^{2\ell+2k-2k_1-2k_2}.
\end{align}
Using this bound and standard H\"older estimates, 
we can reduce \eqref{prored100'} to proving the following:
\begin{equation}
\label{prored200}
\begin{aligned}
&2^{(m - k)(1+\delta)}  \norm{B_{m,\ell}(P_{k_1}f,P_{k_2}f)}_{L^2} \lesssim \e_1^2 2^{-\delta^2 m},
\\
&\text{with } \abs{k_1-k_2}\leq 10,  \quad -2m \leq \ell,k_1,k_2 \leq 4\delta m, \quad -2m \leq k \leq -5\delta m + D.
\end{aligned}
\end{equation}

The rest of this proof is dedicated to showing \eqref{prored200} and split into two cases, depending on which of the parameters $\ell$ or $k$ is smaller.

\medskip
\noindent
{\bf Case 1: $\ell \leq k + 5$.} 
In this case we must have $k \geq \min\{k_1,k_2\} - 15$, so that $k,k_1,k_2$ are all comparable to each other and smaller than $-5\delta m + D$.
In particular \eqref{proredsymbound} gives
\begin{equation}
\label{prored205}
{\| \mathfrak{m}^{k,k_1,k_2} \varphi_\ell(\xi-2\eta) \|}_{S^\infty} \lesssim 2^{2\ell-2k_1}.
\end{equation}
We proceed in three steps.


\medskip
\noindent
{\it Step 1: $\ell - k_1 \leq -\frac{4m}{9}$}.
In this case we use integration by parts in time. We introduce a further localization in the size of the phase $\Phi$
in the bilinear operators $B_{m,\ell}$ defined in \eqref{prored101}. 
More precisely, we write
\begin{align}
\label{prored210}
\begin{split}
& B_{m,\ell}(f,g) = B_{m,\ell,\leq p_0} (f,g) + \sum_{p > p_0} B_{m,\ell,p} (f,g), \qquad  p_0 := -3m,
\\
& B_{m,\ell,\ast}(f,g) := \int_0^t \tau_m(s) \int_{\R^2} \varphi_{\ast}(\Phi(\xi,\eta)) W_{\ell}(f, g)(\xi,\eta) \, d\eta ds.
\end{split}
\end{align}
where $W_{\ell}$ is given in \eqref{prored101}.

Notice that in analyzing the terms in \eqref{prored210} we will be dealing with a kernel of the form
\begin{align}
\label{prored211}
K_{p,\ell}(\xi,\eta) := \varphi_p(\Phi(\xi,\eta)) \varphi_\ell(\xi-2\eta) \wt{\varphi}_k(\xi) \wt{\varphi}_{k_1}(\xi-\eta) \wt{\varphi}_{k_2}(\eta).
\end{align}
Since $k,k_1,k_2$ are all comparable and much larger than $\ell$ we see, using \eqref{eq:Schur_p3} in Lemma \ref{lem:S+set}, that
\begin{equation}
\label{eq:pureSchur_l}
{\|K_{p,\ell} \|}_{Sch} \lesssim 2^{p+\frac{5k_1}{2}+\frac{\ell}{2}}.
\end{equation}
We can directly use this estimate to obtain the desired bound \eqref{prored200} for the term $B_{m,\ell,\leq p_0}$.
Since we must also have $\abs{\Phi} \lesssim 2^{-2k_1} \lesssim 2^{5m}$, there are only $O(m)$ terms in the sum in \eqref{prored210},
and it will thus suffice to prove
\begin{equation}
\label{prored220}
 2^{(m - k)(1+\delta)}  {\| B_{m,\ell,p}(P_{k_1}f,P_{k_2}f) \|}_{L^2} \lesssim \e_1^2 2^{-3\delta^2 m},
\end{equation}
for fixed $p \in [-3m,5m]$.

Integrating by parts in $s$ 
we can write:
\begin{equation}
\label{prored221}
\begin{aligned}
 B_{m,\ell,p}(P_{k_1}f,P_{k_2}f) & = I_{m,\ell,p}(P_{k_1}f,P_{k_2}f) - I\!I_{m,\ell,p}(\partial_t P_{k_1}f,P_{k_2}f) - I\!I_{m,\ell,p}(P_{k_1}f,\partial_t P_{k_2}f),
\\
 I_{m,\ell,p}(f,g) &:= \int_0^t 2^{-m} \tau^\prime_m(s) \int_{\R^2} \frac{\varphi_{p}(\Phi(\xi,\eta))}{i\Phi(\xi,\eta)} W_{\ell}(f,g)(\xi,\eta) \, d\eta ds,
\\
 I\!I_{m,\ell,p}(f,g) &:= \int_0^t \tau_m(s) \int_{\R^2} \frac{\varphi_{p}(\Phi(\xi,\eta))}{i\Phi(\xi,\eta)} W_{\ell}(f,g)(\xi,\eta) \, d\eta ds.
\end{aligned}
\end{equation}

For the first above term, using the a priori bounds \eqref{apriori100}-\eqref{aprioriH-1}, the bound on the symbol \eqref{prored205}
and the bound on the kernel \eqref{prored211}, we have the estimate
\begin{equation*}
\begin{aligned}
&2^{(m-k)(1+\delta)} \norm{I_{m,\ell,p}(P_{k_1}f,P_{k_2}f)}_{L^2}
\\
&\qquad \lesssim 2^{(m-k)(1+\delta)} \cdot 2^{2\ell -2k_1} \cdot 2^{-p} \cdot 
  {\| K_{p,\ell}(\xi,\eta) \what{P_{k_1}f}(\xi-\eta) \|}_{Sch} \norm{P_{k_2} f}_{L^2}
\\
&\qquad  \lesssim 2^{(m-k)(1+\delta)} \cdot 2^{\frac{k_1}{2}+\frac{5\ell}{2}} 2^{-k_1(1+\delta)} \e_1 \cdot 2^{k_2} \e_1
\\
&\qquad  \lesssim 2^{-(\frac{1}{2} + 2\delta)k_1} 2^{(1+\delta)m} 2^{\frac{5\ell}{2}} \e_1^2 \lesssim \e_1^2 2^{-m/40},
\end{aligned}
\end{equation*}
having used the assumption $\ell \leq -\frac{4m}{9} + k_1$ 
for the last step.

For the remaining terms in \eqref{prored221} we can use a similar bound together with \eqref{Lemdsfconc} to obtain
\begin{equation*}
\begin{aligned}
& 2^{(m-k)(1+\delta)} \norm{I\!I_{m,\ell,p} (\partial_t P_{k_1}f,P_{k_2}f)}_{L^2}
\\
& \qquad \lesssim 2^{(m-k)(1+\delta)} \cdot 2^m \cdot 2^{\frac{k_1}{2}+\frac{5\ell}{2}} \cdot
  {\| \what{P_{k_1}f} \|}_{L^\infty} \sup_{s\approx 2^m} \norm{\partial_s P_{k_2}f}_{L^2}
\\
& \qquad \lesssim 2^{-(1+\delta)k} 2^{(2+\delta)m} \cdot 2^{\frac{k_1}{2}+\frac{5\ell}{2}} \cdot \e_1 2^{-(1+\delta)k_1} \cdot \e_1^2 2^{k_2} 2^{-2m+10\delta m}
\\
& \qquad \lesssim 2^{11\delta m} 2^{-(\frac{1}{2}+2\delta)k_1} 2^{\frac{5\ell}{2}} \e_1^3 \lesssim \e_1^3 2^{-m/40}.
\end{aligned}
\end{equation*}
The same bound can be similarly obtained for $I\!I_{m,p}(P_{k_1}f, \partial_t P_{k_2}f)$ and this concludes the proof of \eqref{prored220}
when $\ell - k_1 \leq -\frac{4m}{9}$.

To deal with the remaining cases 
we introduce the usual spatial localizations as defined in \eqref{prored20}, and aim to show
\begin{equation*}
 2^{(m - k)(1+\delta)} \sum_{j_1,j_2} {\| B_{m,\ell}(f_1,f_2) \|}_{L^2} \lesssim \e_1^2 2^{-2\delta^2 m},
\end{equation*}
under the assumptions in \eqref{prored200} and with $\ell - k_1 \geq -\frac{4m}{9}$.

\medskip
\noindent
{\it Step 2: $\ell - k_1 \geq -\frac{4m}{9}$ and $\max\{j_1,j_2\} \leq m+\ell - 3k_1 - \delta m$}.
In this case we can repeatedly integrate by parts. 
Indeed, in our current frequency configuration we have $|\nabla_\eta \Phi| \approx 2^\ell 2^{-3k_1}$, see \eqref{nabla_xiPhi}.
Then we can use Lemma \ref{lemIBP0} by letting $K = s (2^\ell 2^{-3k_1})^{-1}$, $F(\eta) = \Phi 2^\ell 2^{-3k_1}$ and $\epsilon = 2^{\ell}$.
From \eqref{IBP02}, choosing $M$ large enough, we then obtain $ {\| B_{m,\ell}(f_1,f_2) \|}_{L^2} \lesssim 2^{-10m} {\| f_1\|}_{L^2}{\| f_2\|}_{L^2}$,
which is more than sufficient to obtain \eqref{prored200}.

\medskip
\noindent
{\it Step 3: $\max\{j_1,j_2\} \geq m+\ell - 3k_1 - \delta m$}.
In this case a standard H\"older estimate, placing the input with largest position in $L^2$, suffices:
\begin{align*}
& 2^{(m-k)(1+\delta)} \norm{B_{m,\ell}(f_1,f_2)}_{L^2}
\\
& \qquad \lesssim 2^{(m-k)(1+\delta)} \cdot 2^m \cdot 2^{2\ell -2k_1}
  \cdot 2^{-m} 2^{(2-\delta)k_1}2^{-4k_1^+} \e_1 \cdot 2^{-\max\{j_1,j_2\}} 2^{-(1+\delta)k_2} \e_1 \cdot 2^{-\delta (j_1+j_2)}
\\
& \qquad \lesssim 2^{2\delta m} 2^{\ell} 2^{(1-3\delta)k_1}2^{-4k_1^+} 2^{-\delta(j_1+j_2)} \e_1^2,
\end{align*}
having used the a priori bounds \eqref{apriori100}-\eqref{apriori101}, and the symbol bound \eqref{prored205}.
Summing over $j_1,j_2$ we see that this implies the desired bound \eqref{prored200}
since  $\min\{k,k_1,k_2\} \leq -5\delta m + D$ holds.

\begin{remark}
\label{remestl}
Notice that the bounds proved above suffice to obtain an estimate as in \eqref{eq:weightbound'} for $\sum_\ell B_{m,\ell}$ instead of $B_m$,
provided that $\ell \leq -5\delta m$, and placing no additional smallness restriction on $k$.
\end{remark}

\medskip
\noindent
{\bf Case 2: $k \leq \ell-5$.} 
Here we have $k \leq -5\delta m + D$ and $\abs{\ell - k_1} \leq 20$, and similar arguments to those of Case 1
can be used essentially by reversing the roles of $k$ and $\ell$.
Note that in this case stronger bounds are available for the kernel that we need to consider, see \eqref{eq:pureSchur_k} below.
We decompose the profiles according to their spatial localization as done above and proceed as follows.

\medskip
\noindent
{\it Step 1: $\max\{j_1,j_2\} \leq m + k- 3k_1 - \delta m$}.
Note that this case will be empty if $k < -m+3k_1 + \delta m$ and only Step 2 below needs to be performed.
In the current scenario we have $|\nabla_\eta \Phi| \approx 2^{k-3k_1}$ and $|D^\alpha_\eta \Phi| \lesssim 2^{k} 2^{-(|\alpha|+2)k_1}$, $|\alpha| \geq 1$.
We can then use Lemma \ref{lemIBP0} by letting $K = s (2^k 2^{-3k_1})^{-1}$, $F(\eta) = \Phi 2^k 2^{-3k_1}$ and $\epsilon = 2^{k_1}$,
obtaining ${\| B_{m,\ell}(f_1,f_2) \|}_{L^2} \lesssim 2^{-10m} {\| f_1\|}_{L^2}{\| f_2\|}_{L^2}$.

\medskip
\noindent
{\it Step 2: $\max\{j_1,j_2\} \geq m + k - 3k_1 - \delta m$}.
In this case we want to use integration by parts in $s$ similarly to Step 1 of Case 1 above.
From the formula for the symmetrized symbol we see that the bound \eqref{prored205} used before
can be substituted by
\begin{align*}
{\| \mathfrak{m}^{k,k_1,k_2} \varphi_\ell(\xi-2\eta) \|}_{S^\infty} \lesssim 2^{2k-2k_1}.
\end{align*}
Moreover, notice that we have a bound stronger than \eqref{eq:pureSchur_l} for the relevant kernel, that is
\begin{align}
\label{eq:pureSchur_k}
{\| \varphi_p(\Phi(\xi,\eta)) \varphi_\ell(\xi-2\eta) \wt{\varphi}_k(\xi) \wt{\varphi}_{k_1}(\xi-\eta) \wt{\varphi}_{k_2}(\eta) \|}_{Sch}
  \lesssim 2^{p+k+2k_1},
\end{align}
as per \eqref{eq:Schur_p3} in Lemma \ref{lem:S+set}.
Then the same arguments as in Step 1 of Case 1 above go through and give the main conclusion \eqref{eq:weightbound} 
when $k  \leq \min\{k_1, -5\delta m \}+D$. This concludes the proof of the Proposition.
\end{proof}

\medskip
As a consequence of Proposition \ref{prop:red_deltas} we have the following:
\begin{cor}
In order to prove the main bound \eqref{eq:w_bound} it will be enough to prove the following claim:
for all $(k,j) \in \mathcal{J}$ we have
\begin{equation}
\label{eq:w_bound1}
\begin{aligned}
& 2^{4k_+} 2^{m-2\min\{k,k_1,k_2\}+k}
 \norm{P_k B_{m,\ell}(P_{k_1}f,P_{k_2}f)}_{L^2} \lesssim 2^{-2\delta m} \e_1^2,  
\\ & \text{whenever} \quad -5\delta m \leq k,k_1,k_2,\ell \leq 4\delta m + D^2,
\end{aligned}
\end{equation}
where $B_{m,\ell}$ is defined as 
\begin{equation}
\label{prored101'}
\begin{aligned}
& \mathcal{F} B_{m,\ell}(f,g) = \int_0^t \tau_m(s) \int_{\R^2} W_{\ell}(f,g)\, d\eta ds,
\\
& W_{\ell}(f,g)(\xi,\eta) := e^{is\Phi(\xi,\eta)} \mathfrak{m}(\xi,\eta) \varphi_\ell(\xi-2\eta)\what{f}(\xi-\eta)\what{g}(\eta).
\end{aligned}
\end{equation}
\end{cor}

\begin{proof}
In view the estimates \eqref{eq:weightbound}, \eqref{eq:weightbound'} in Proposition \ref{prop:red_deltas},
we know that to obtain the main bound \eqref{eq:w_bound} it will suffice to show
\begin{align}
\label{eq:w_bound3}
\begin{split}
\sup_{\substack{k+j\geq 0\\ k \geq -5\delta m}} & 2^{4k_+} 2^{(k+j)(1+\delta)}
 \sum_{k_1,k_2 \geq -5\delta m} \norm{Q_{jk} B_{m}(P_{k_1}f,P_{k_2}f)}_{L^2} \lesssim 2^{-\delta^3 m} \e_1^2.
\end{split}
\end{align}
Recall that from \eqref{Basiccasescons'} we have the upper bound $\max\{k,k_1,k_2\} \leq \delta(j+m) + D$.
Then the finite speed of propagation Lemma \ref{FSlemma} suffices to bound the sum in \eqref{eq:w_bound3}
whenever $j \geq m - 2\min\{k,k_1,k_2\} + D$.
We may therefore restrict ourselves to $j \leq m - 2\min\{k,k_1,k_2\} + D \leq (1+10\delta)m + D$, and thus also to $\max\{k,k_1,k_2\} \leq 4\delta m +D$.
We then have a sum over at most $O(m^2)$ terms so that it suffices to prove the bound
\begin{align*}
\begin{split}
2^{4k_+} 2^{(k+j)(1+\delta)} \norm{Q_{jk} B_{m}(P_{k_1}f,P_{k_2}f)}_{L^2} \lesssim 2^{-(3/2)\delta^3 m} \e_1^2
\end{split}
\end{align*}
for each fixed triple $k,k_1,k_2$ with $-5\delta m \leq k,k_1,k_2 \leq 4\delta m + D$, and $(k,j) \in \mathcal{J}$.
Moreover, in view of Remark \ref{remestl} we may also replace $B_m$ above with $B_{m,\ell}$ and assume that $\ell \geq -5\delta m$.
The claim follows since $\delta(m - 2\min\{k,k_1,k_2\} + k) \leq (3/2)\delta m$.
\end{proof}

\medskip
\vskip10pt
\subsection{Further Reductions}
We now turn to further reductions on the size of the phase $\Phi$ and the spatial localization of the profiles
in the bilinear term $B_{m,\ell}(P_{k_1}f,P_{k_2}f)$ in \eqref{prored101'}.
For this purpose let us write
\begin{align}
\label{fred1.1}
& B_{m,\ell}(P_{k_1}f,P_{k_2}f) = \sum_{p \in \Z} B_{m,\ell,p}(P_{k_1}f,P_{k_2}f) = \sum_{r,p \in \Z} B_{m,\ell,r,p}(P_{k_1}f,P_{k_2}f),
\\
\label{fred1.2}
& B_{m,\ell,p}(f,g) := \mathcal{F} \int_0^t \tau_m(s) \int_{\R^2} \varphi_p(\Phi(\xi,\eta)) W_{\ell}(f,g)\, d\eta ds,
\\
\label{fred1.3}
& B_{m,\ell,r,p}(f,g) := \mathcal{F} \int_0^t \tau_m(s) \int_{\R^2} \varphi_p(\Phi(\xi,\eta)) \varphi_r(\eta-2\xi) W_{\ell}(f,g)\, d\eta ds.
\end{align}
where $W_\ell$ is as in \eqref{prored101'}. 
Notice that $B_{m,\ell,p}(P_{k_1}f,P_{k_2}f)$ is trivial unless $p \leq -2\min\{k,k_1,k_2\} + D \leq 10\delta m + D$ and
$r \leq \max\{k_1,k_2\} + D \leq 4\delta m + 2D^2$.
Also note that a Schur-type estimate using Lemma \ref{lem:S+set} will give the desired bound for the sum of the terms
$B_{m,\ell,p}$ when $p \leq -3m$. Similarly, it is not hard to see that one can obtain the bound \eqref{eq:w_bound1} for
the terms $B_{m,\ell,p,r}$ when $r \leq -3m$.
Therefore the summations in \eqref{fred1.1} are all over at most $O(m^2)$ terms, and it suffices to prove the bound for each element in the sum.

\begin{prop}\label{propfred}
With the usual notation
$f_\nu = P_{[k_\nu-2,k_\nu+2]} \varphi_{j_\nu}^{(k_\nu)}(x)P_{k_\nu} f, j_\nu+k_\nu \geq 0, \nu=1,2$,
and under the frequency restriction in \eqref{eq:w_bound1}, namely
\begin{equation*}
-5\delta m \leq k,k_1,k_2,\ell \leq 4\delta m + D^2,
\end{equation*}
we have
\begin{align}
\label{propfred1}
\begin{split}
 \norm{P_k B_{m,\ell}(f_1,f_2)}_{L^2} \lesssim 2^{-2m} \e_1^2
\quad \text{if }\; \max\{ j_1,j_2\} \leq m + \min\{k,\ell\} - 3k_1 -\delta m.
\end{split}
\end{align}
If instead $\max\{ j_1,j_2\} \geq m + \min\{k,\ell\} - 3k_1 -\delta m$, then we have the following bounds:
\begin{alignat}{2}
\label{propfred2}
& 2^{4k_+} 2^{m-2\min\{k,k_1,k_2\}+k} \norm{ P_k B_{m,\ell,p}(f_1,f_2) }_{L^2} \lesssim 2^{-3\delta m} \e_1^2 \quad &&\text{if }\; p \geq  -m +40\delta m,
\\ \nonumber
\\
\label{propfred3}
& 2^{4k_+} 2^{m-2\min\{k,k_1,k_2\}+k} \norm{ P_k B_{m,\ell,p,r}(f_1,f_2) }_{L^2} \lesssim 2^{-4\delta m} \e_1^2 \quad &&\text{if }\; r \leq - 35\delta m,
\\ \nonumber
\\
\label{propfred4}
& 2^{4k_+} 2^{m-2\min\{k,k_1,k_2\}+k} \norm{ P_k B_{m,\ell,p,r}(f_1,f_2) }_{L^2} \lesssim 2^{-4\delta m} \e_1^2 \quad &&\text{if }
  \; \min\{ j_1,j_2\} \geq \frac{m}{2} + 60\delta m.
\end{alignat}
\end{prop}

For convenience we introduce the notation
\begin{equation}
\label{notmaxmin}
\ul{k} := \min\{k_1,k_2\}, \quad \ol{k} := \max\{k_1,k_2\}, \qquad \ul{j} := \min\{j_1,j_2\}, \quad \ol{j} := \max\{j_1,j_2\}.
\end{equation}

\begin{proof}
Each one of the bounds in the statement can be proven via similar techniques to those used in the proof of Proposition \eqref{prop:red_deltas} above.

\medskip
\noindent
{\it Proof of \eqref{propfred1}}.
This follows by integrating by parts in $\eta$ sufficiently many times, i.e.\ by applying Lemma \ref{lemIBP0}
using the fact that $|\nabla_\eta \Phi| \approx 2^{k+\ell-4k_1}$ and $|D^\alpha_\eta \Phi| \lesssim 2^{-(|\alpha|+1) \min\{k_1,k_2\}}$
on the support of the integral.

\medskip
\noindent
{\it Proof of \eqref{propfred2}}.
Now we treat the term $B_{m,\ell,p}$ as defined in \eqref{fred1.2} analogously to what was done in \eqref{prored210} 
and integrate by parts in $s$.
Similarly to \eqref{prored221} we obtain
$B_{m,\ell,p}(f_1,f_2) = I_{m,p}(f_1,f_2) - I\!I_{m,p}(\partial_tf_1,f_2) - I\!I_{m,p}(f_1,\partial_tf_2)$ where
\begin{align}
\label{propfred10}
\begin{split}
& I_{m,\ell,p}(f,g) := \int_0^t 2^{-m} \tau^\prime_m(s) \int_{\R^2} \frac{\varphi_{p}(\Phi(\xi,\eta))}{i\Phi(\xi,\eta)} W_\ell(f,g)(\xi,\eta) \, d\eta ds,
\\
& I\!I_{m,\ell,p}(f,g) := \int_0^t \tau_m(s) \int_{\R^2} \frac{\varphi_{p}(\Phi(\xi,\eta))}{i\Phi(\xi,\eta)} W_\ell(f,g)(\xi,\eta) \, d\eta ds.
\end{split}
\end{align}

For the first term in \eqref{propfred10} we use Lemma \ref{lem:H+ibpt} and the a priori bounds,
estimating the profile with the largest spatial localization in $L^2$ and obtain
\begin{align*}
{\| P_k I_{m,\ell,p}(f_1,f_2) \|}_{L^2} & \lesssim 2^{-p} \cdot {\|\mathfrak{m}^{k,k_1,k_2} \varphi_\ell(\xi-2\eta) \|}_{S^\infty} \cdot 2^{-m} 2^{-2\ol{k}^+} \e_1
  \cdot 2^{-\ol{j}} 2^{-\ul{k}} \e_1
\\
& \lesssim 2^{-m - 39\delta m} \cdot {\|\mathfrak{m}^{k,k_1,k_2} \|}_{S^\infty}
  \cdot 2^{-\ul{k}-2\ol{k}^+} 2^{-\min\{k,\ell\} + 3k_1} \e_1^2.
\end{align*}
Using the bound ${\|\mathfrak{m}^{k,k_1,k_2} \|}_{S^\infty} \lesssim 2^{-\ul{k} + \ol{k}}$, we see that
\begin{align*}
2^{m+k-2\min\{k,k_1,k_2\}} \norm{P_k I_{m,\ell,p}(f_1,f_2)}_{L^2} & \lesssim \e_1^2 2^{-39\delta m} \cdot 2^{-4\min\{0,k,k_1,k_2,\ell\}}2^{2\max\{0,k,k_1,k_2,\ell\}}, 
\end{align*}
which suffices to obtain \eqref{propfred2} in view of the restrictions in \eqref{eq:w_bound1}.

For the terms $I\!I_{m,p}$ we use Lemma \ref{lem:H+ibpt}, estimating in $L^2$ the term involving the time derivative of the profile
via \eqref{Lemdsfconc}, together with the bound for the symbol used above:
\begin{align*}
{\| P_k I\!I_{m,\ell,p}(f_1,f_2) \|}_{L^2} & \lesssim  2^{m} \cdot 2^{-p}\cdot {\| \mathfrak{m}^{k,k_1,k_2} \varphi_\ell(\xi-2\eta) \|}_{S^\infty}
  \cdot 2^{-m} 2^{2\ol{k}} \e_1 \cdot 2^{\ul{k}} 2^{-2m + 10\delta m} \e_1 2^{-4\ol{k}^+}
\\
& \lesssim 2^{-m - 30\delta m} \cdot 2^{3\ol{k} - 4\ol{k}^+} \e_1^2.
\end{align*}
This suffices to prove \eqref{propfred2}.

\medskip
\noindent
{\it Proof of \eqref{propfred3}}.
We now look at the bilinear term $B_{m,\ell,p,r}$ defined in \eqref{fred1.3}
with $r \leq -35\delta m \leq \min\{k,k_1,k_2\} - D$, so that $k,k_1,k_2$ and $\ell$ are all comparable.
In view of the previous step we may assume $p \leq -m + 35\delta m$.
Using the estimate \eqref{eq:Schur_p2} in Lemma \ref{lem:S+set}\eqref{it:Schur_p2} 
we see that
\begin{align*}
 \norm{\varphi_p(\Phi(\xi,\eta)) \wt{\varphi}_k(\xi) \wt{\varphi}_{k_1}(\xi-\eta) \wt{\varphi}_{k_2}(\eta)
  \wt{\varphi}_\ell(\xi-2\eta) \wt{\varphi}_r(\eta-2\xi)}_{Sch} \lesssim 2^{p + \frac{r}{2} + \frac{5}{2}k}.
\end{align*}
Using this bound with Schur's test, $|\mathfrak{m}^{k,k_1,k_2}| \lesssim 2^{r-k}$, $\ol{j} \geq (1-\delta) m - 2k$,
and the usual a priori bounds, we see that
\begin{equation*}
\begin{aligned}
& 2^{4k^+} 2^{m-2\min\{k,k_1,k_2\}+k} \norm{P_k B_{m,\ell,p,r}(f_1,f_2)}_{L^2}
\\
& \qquad \lesssim 2^{m-k} \cdot 2^m\cdot 2^{p + \frac{r}{2} + \frac{5}{2}k} \cdot 2^{r-k} \cdot 2^{-k}\e_1 \cdot 2^{-\ol{j}} 2^{-k} \e_1
\\
& \qquad \lesssim 2^{m+\delta m} 2^{p + \frac{3}{2}r + \frac{1}{2}k} \e_1^2,
\end{aligned}
\end{equation*}
which is sufficient to obtain \eqref{propfred3}.

\medskip
\noindent
{\it Proof of \eqref{propfred4}}.
In view of the previous step we may assume $p \leq -m +40\delta m$ and $r \geq -35\delta m$.
Just for the purpose of this proof let us define
\begin{align*}
K(\xi,\eta) := \varphi_p(\Phi(\xi,\eta)) \varphi_\ell(\xi-2\eta) \varphi_r(\xi-2\eta) 
  \wt{\varphi}_k(\xi) \wt{\varphi}_{k_1}(\xi-\eta) \wt{\varphi}_{k_2}(\eta).
\end{align*}
In view of Lemma \ref{lem:S+set}\eqref{it:Schur_p2} we have, recall the notation \eqref{notmaxmin}, 
${\| K(\xi,\eta) \|}_{Sch} + {\| K(\xi,\xi-\eta) \|}_{Sch} \lesssim 2^{p + (1/2)\ul{k} + (3/2)\ol{k}}$.
Also notice that for any kernel with $\abs{K} \lesssim 1$ one has $\norm{K(\xi,\eta) g(\xi-\eta)}_{Sch}
  \lesssim \norm{K(\xi,\eta) }_{Sch}^{1/2} \norm{g}_{L^2}$.
Then, using Schur's test by estimating in $L^2$ the profile corresponding to the larger localization $2^{\ol{j}}$ we can bound
\begin{align*}
\begin{split}
 \norm{P_k B_{m,\ell,p,r}(f_1,f_2) }_{L^2} & \lesssim 2^m \cdot \big( 2^{p + \frac{1}{2}\ul{k} + \frac{3}{2}\ol{k}} \big)^{\frac{1}{2}}
  \cdot {\| \mathfrak{m}^{k,k_1,k_2} \|}_{S^\infty}  \cdot  {\| f_{\ul{j}} \|}_{L^2} \cdot {\| f_{\ol{j}} \|}_{L^2}
\\
& \lesssim 2^m \cdot 2^{\frac{p}{2} + \frac{3}{4}\ol{k} + \frac{1}{4}\ul{k}} \cdot 2^{\ol{k}-\ul{k}} \cdot 2^{-\ol{j} -\ul{j}} \cdot 2^{-\ol{k} - \ul{k}-4\ol{k}^+}\e_1^2.
\end{split}
\end{align*}
Using the assumptions $p \leq -m +40\delta m$, $\ol{j} \geq (1-\delta)m - 3k_1 + \min\{k,\ell\}$ and $\ul{j} \geq \frac{1}{2}m + 60\delta m$, 
we see that
\begin{align*}
\begin{split}
2^{4k^+} 2^{m-2\min\{k,k_1,k_2\}+k} {\| P_k B_{m,\ell,p,r}(f_1,f_2) \|}_{L^2}
& \lesssim \e_1^2 2^{-39\delta m} \cdot 2^{k-2\min\{k,k_1,k_2\}} \cdot 2^{\frac{3}{4}\ol{k} - \frac{7}{4}\ul{k} } \cdot 2^{3k_1 - \min\{k,\ell\}}
\\
& \lesssim \e_1^2 2^{-39\delta m} \cdot 2^{-\frac{15}{4}\min\{0,k,k_1,k_2,\ell\} + \frac{15}{4}\max\{0,k_1,k_2\}},
\end{split}
\end{align*}
which is sufficient for \eqref{propfred4}, again in view of \eqref{eq:w_bound1}.
\end{proof}

\vskip20pt
\section{The Weighted Estimate: Part II}\label{sec:weightII}

Recall that the main weighted bound \eqref{eq:w_bound} is implied by \eqref{eq:w_bound1}.
Combining this fact with the estimates in Proposition \ref{propfred} we can reduce
the proof of the main desired bound to showing that
\begin{equation}
\label{estweightII}
 2^{4k^+} 2^{m-2\min\{k,k_1,k_2\}+k} \norm{ P_k B_{m,\ell,\leq p_0,r}(f_1,f_2) }_{L^2} \lesssim 2^{-4\delta m},
\end{equation}
where
\begin{equation*}
\begin{aligned}
B_{m,\ell,\leq p_0,r}(f,g) & := \mathcal{F}^{-1}
    \int_0^t \tau_m(s)\int_{\R^2} e^{is\Phi(\xi,\eta)} \varphi_{\leq p_0}(\Phi(\xi,\eta))
    \\ 
    &\qquad \times \mathfrak{m}(\xi,\eta) \, \varphi_\ell(\xi-2\eta) \varphi_r(2\xi-\eta)\, \hat{f}(\xi-\eta)\hat{g}(\eta)\, d\eta,
\end{aligned}
\end{equation*}
and whenever
\begin{equation}
\label{restII}
\begin{aligned}
& -5\delta m \leq k,k_1,k_2,\ell \leq 4\delta m + D^2,\quad r \geq -35\delta m,
\\
& p_0 := -m +40\delta m,
\\
& \ol{j} := \max\{ j_1,j_2\} \geq m + \min\{k,\ell\} - 3k_1 -\delta m \geq m -20\delta m ,
\\
& \ul{j} := \min\{ j_1,j_2\} \leq \frac{m}{2} + 60\delta m.
\end{aligned}
\end{equation}

\begin{remark}
Intuitively speaking the reductions to the configuration \eqref{restII} have placed us in a framework where neither integration by parts in time nor space produces any gain:
$\abs{\Phi}$ is of the order of $s^{-1}$ and $\abs{\nabla_\eta\Phi}$ is of order about $1$, with $\ol{j}$ of the order about $s$.
Notice that this is not a localization to, but rather away from the resonant set.
\end{remark}

\vskip10pt
\subsubsection*{{\bf Anisotropic Decomposition}}
We now decompose the bilinear term into two pieces, according to the size of $|\xi_1-\eta_1|$:
\begin{equation}
\label{B_q}
\begin{aligned}
& B_{m,\ell,\leq p_0,r}(f_1,f_2) = \mathcal{B}_{\leq q_0}(f_1,f_2) + \sum_{q > q_0} \mathcal{B}_{q}(f_1,f_2) , 
\qquad \qquad  q_0:= -\frac{m}{20},
\\
& \mathcal{B}_{\ast}(f,g) := \mathcal{F}^{-1}
    \int_0^t \tau_m(s)\int_{\R^2} e^{is\Phi(\xi,\eta)} \varphi_{\leq p_0}(\Phi(\xi,\eta))  \varphi_\ast(\xi_1-\eta_1)
    m_{\ell,r}(\xi,\eta) \, \hat{f}(\xi-\eta)\hat{g}(\eta)\, d\eta,
\\
\\
& m_{\ell,r}(\xi,\eta) := \mathfrak{m}^{k,k_1,k_2}(\xi,\eta) \varphi_\ell(\xi-2\eta) \varphi_r(2\xi-\eta),
\end{aligned}
\end{equation}
see also the notation \eqref{mloc}, and recall the formula \eqref{FSBsym} for the symbol $\mathfrak{m}$.
Note that in order to simplify notation we suppress the dependence on $m,\ell,p_0,r$ in $\mathcal{B}_{\ast}$.

\vskip10pt
\subsection{Estimate of $\mathcal{B}_{\leq q_0}$}

Here we show how we can exploit the smallness in the localization in $|\xi_1-\eta_1|$ to close our bounds.
The main tool here is given by improved Schur kernel bounds.

Let us introduce the notation
\begin{align*}
K_{q_0}(\xi,\eta) := \varphi_{\leq p_0}(\Phi(\xi,\eta))  \varphi_{\leq q_0} (\xi_1-\eta_1) m_{\ell,r}(\xi,\eta),
\end{align*}
where $m_{\ell,r}$ is as in \eqref{B_q}, and so that
\begin{align*}
\begin{split}
& \mathcal{B}_{\leq q_0}(f,g) = \mathcal{F}^{-1} \int_0^t \tau_m(s) \int_{\R^2} e^{is\Phi(\xi,\eta)} K_{q_0}(\xi,\eta) \, \hat{f}(\xi-\eta)\hat{g}(\eta)\, d\eta.
\end{split}
\end{align*}

\begin{prop}\label{prop:qsmall}
Under the assumptions \eqref{restII} the following holds true:
\begin{align}
\label{estweightqsmall}
\begin{split}
& 2^{4k^+} 2^{m-2\min\{k,k_1,k_2\}+k} \norm{P_k \mathcal{B}_{\leq q_0}(f_1,f_2) }_{L^2} \lesssim 2^{-4\delta m}.
\end{split}
\end{align}
\end{prop}

\begin{proof}
Observe that
\begin{align*}
2^{p_0} \gtrsim |\Phi(\xi,\eta) | =
  \Big| (\xi_1-\eta_1)\left(\frac{1}{|\eta|^2}-\frac{1}{|\xi-\eta|^2}\right)-\eta_1\left(\frac{1}{|\eta|^2}-\frac{1}{|\xi|^2}\right) \Big|.
\end{align*}
Since on the support of the integral \eqref{B_q} we have $|\xi_1-\eta_1|\leq 2^{q_0}$, we see that
\begin{equation}
\label{qsmall2}
\abs{\eta_1} \abs{\frac{1}{\abs{\eta}^2}-\frac{1}{\abs{\xi}^2} } \lesssim 2^{p_0} + 2^{q_0} \abs{\frac{1}{\abs{\eta}^2}-\frac{1}{\abs{\xi-\eta}^2} }
  \lesssim 2^{q_0 + 10\delta m}.
\end{equation}
We then distinguish two main cases depending on the size of $\abs{\eta_1}$ relative to $2^{\frac{q_0}{3} + 10\delta m}$.
More precisely we write
\begin{align*}
\begin{split}
& \mathcal{B}_{\leq q_0}(f,g) = \mathcal{B}_{\leq q_0}^{-}(f,g) + \mathcal{B}_{\leq q_0}^{+}(f,g),
\\
& \mathcal{B}_{\leq q_0}^{\pm}(f,g) := \mathcal{F}^{-1} \int_0^t \tau_m(s) \int_{\R^2} e^{is\Phi(\xi,\eta)} K_{q_0}(\xi,\eta) \chi_\pm(\eta_1)
  \, \what{f}(\xi-\eta)\what{g}(\eta) \, d\eta,
\\
& \chi_-(\eta_1) := \varphi_{\leq \frac{q_0}{3} + 10\delta m}(\eta_1), \quad \chi_+(\eta_1) := 1 - \chi_-(\eta_1).
\end{split}
\end{align*}

\medskip
\noindent
\textit{Estimate of $\mathcal{B}_{\leq q_0}^-$}.
In this case $|\eta_1| \lesssim 2^{\frac{q_0}{3} + 10\delta m}$ and we see that
\begin{align*}
|\xi\cdot\eta^\perp| \lesssim |(\xi_1-\eta_1) \eta_2| + |(\xi_2-\eta_2) \eta_1| \lesssim 2^{\frac{q_0}{3} + 15\delta m}.
\end{align*}
This gives us an improved estimate on the symbol $\mathfrak{m}$, see \eqref{FSBsym}, 
and hence on the kernel: Using 
Lemma \ref{lem:S+set}\eqref{it:Schur_p2} and the restrictions \eqref{restII} 
we see that
\begin{align*}
\norm{K_{q_0}(\xi,\eta)\chi_-(\eta_1)}_{Sch} + \norm{K_{q_0}(\xi,\xi-\eta)\chi_-(\xi_1-\eta_1)}_{Sch}
  & \lesssim 2^{\frac{q_0}{3}}  \cdot 2^{p_0} \cdot  2^{40\delta m}.
\end{align*}
We then apply Schur's test incorporating the profile with localization $\ul{j}$ in the kernel 
and estimating the one with largest $\ol{j}$ in $L^2$:
Using the a priori bounds \eqref{apriori100} and \eqref{apriori102}
together with the restrictions \eqref{restII} we have
\begin{equation*}
\begin{aligned}
{\| \mathcal{B}_{\leq q_0}^-(f_1,f_2)\|}_{L^2}
  &\lesssim 2^m \cdot 2^{\frac{q_0}{3} + p_0 + 40\delta m} \cdot \e_1 2^{5\delta m} \cdot \e_1 2^{-m + 25 \delta m}
\\
  &\lesssim 2^{-m} \cdot 2^{-\frac{m}{60}} \cdot 2^{110 \delta m} \cdot \e_1^2.
\end{aligned}
\end{equation*}
This is sufficient to obtain \eqref{estweightqsmall},
given that the restrictions \eqref{restII} imply 
$2^{m-2\min\{k,k_1,k_2\}+k} \leq 2^{m} 2^{15\delta m}$ and $\delta \leq 2 \cdot 10^{-4}$.

\medskip
\noindent
\textit{Estimate of $\mathcal{B}_{\leq q_0}^+$}.
In this case $|\eta_1| \gtrsim 2^{\frac{q_0}{3} + 10\delta m}$ and in view of \eqref{qsmall2} 
we must have $\abs{\abs{\eta}^{-2} - \abs{\xi}^{-2}} \leq 2^{\frac{2}{3}q_0}$.
Since $|\eta|^{-2} - |\xi|^{-2} = |\xi|^{-2} |\eta|^{-2} (\xi_2^2 - \eta_2^2 + \xi_1^2 - \eta_1^2)$ we see that
\begin{align*}
|\xi_2^2 - \eta_2^2| \lesssim |\xi|^{2} |\eta|^{2} 2^{\frac{2}{3}q_0} + |\xi_1^2 - \eta_1^2| \lesssim 2^{\frac{q_0}{2} + 16\delta m}.
\end{align*}
Therefore we know that on the support of the integral
\begin{align*}
\begin{split}
& |\xi_1-\eta_1| \lesssim 2^{q_0}, \qquad |\xi_2^2 - \eta_2^2| \lesssim 2^{\frac{2}{3}q_0 + 16\delta m}, \qquad
|\nabla_\xi \Phi(\xi,\eta)|, |\nabla_\eta \Phi(\xi,\eta)| \geq 2^{-50\delta m}, 
\end{split}
\end{align*}
see \eqref{nabla_xiPhi} and the restrictions \eqref{restII}.
Using these we claim that we can estimate
\begin{align}
\label{qsmall30} 
 \norm{K_{q_0}(\xi,\eta)\chi_+(\eta_1)}_{Sch} + \norm{ K_{q_0}(\xi,\xi-\eta)\chi_+(\xi_1-\eta_1) }_{Sch}
  \lesssim 2^{\frac{q_0}{6}}  \cdot 2^{p_0} \cdot  2^{70\delta m}.
\end{align}

To see why this holds true first observe that for the support of the kernel we have
\begin{align*}
{\supp} (K_{q_0}(\xi,\eta)) \subseteq \big\{ (\xi,\eta) \in \R^2\times\R^2 \, : \, \eta \in S^+(\xi) \cup S^-(\xi) \big\},
\end{align*}
where
\begin{align*}
S^\pm(\xi) := \big\{ \eta \in \R^2 \, : \, \abs{\Phi(\xi,\eta)} \lesssim 2^{p_0}, \quad \abs{\nabla_\eta \Phi(\xi,\eta)}, 
\abs{\nabla_\xi \Phi(\xi,\eta)} \gtrsim 2^{-50\delta m},
  \\ \, \abs{\eta_1 - \xi_1 } \lesssim 2^{q_0}, \quad \abs{\eta_2 \pm \xi_2} \lesssim 2^{\frac{q_0}{3} + 8\delta m} \big\}.
\end{align*}
From this observation, and arguments similar to the ones in Lemma \ref{lem:S+set}\eqref{it:Schur_p1}, it follows that
\begin{align*}
\sup_{\xi \in \R^2} \int_{\R^2} \abs{K_{q_0}(\xi,\eta) \chi_+(\eta_1)} \, d\eta \lesssim 2^{p_0 + 60\delta m} \cdot 2^{\frac{q_0}{3} + 8\delta m},
\end{align*}
having also used $|\mathfrak{m}| \lesssim 2^{10\delta m}$.
The same bound can be also deduced for $K_{q_0}(\xi,\xi-\eta) \chi_+(\xi_1-\eta_1)$.
Combing these bounds with the similar but cruder estimate
\begin{align*}
\sup_{\eta \in \R^2} \Big( \int_{\R^2} \abs{K_{q_0}(\xi,\eta)} \, d\xi + \int_{\R^2} \abs{K_{q_0}(\xi,\xi-\eta)} \, d\xi \Big)
  \lesssim 2^{p_0 + 65\delta m}
\end{align*}
we see that \eqref{qsmall30} follows.

We finally use \eqref{qsmall30} and Schur's test to obtain
\begin{equation*}
\begin{aligned}
\norm{\mathcal{B}_{\leq q_0}^+(f_1,f_2)}_{L^2} 
  &\lesssim 2^m \cdot 2^{\frac{q_0}{6} + p_0 + 70\delta m} \cdot \e_1 2^{5\delta m}\cdot \e_1 2^{-m + 25 \delta m}
\\
  &\lesssim 2^{-m} 2^{-\frac{m}{120}} \cdot 2^{140 \delta m} \e_1^2.
\end{aligned}
\end{equation*}
We can then conclude as before, since $\delta$ is small enough.
This suffices to prove the desired bound \eqref{estweightqsmall} and concludes the proof of the Proposition.
\end{proof}

\vskip10pt
\subsection{Estimates of the Terms $\mathcal{B}_q$}

In view of the decomposition \eqref{B_q} and Proposition \ref{prop:qsmall}, the main bound \eqref{estweightII} can be reduced to showing
\begin{align}
\label{estweightIIb}
\begin{split}
& 2^{4k^+} 2^{m-2\min\{k,k_1,k_2\}+k} \norm{P_k \mathcal{B}_q(f_1,f_2) }_{L^2} \lesssim 2^{-5\delta m}, \qquad q\geq q_0,
\end{split}
\end{align}
under the restrictions \eqref{restII}.
This bound can in turn be reduced to the proof of the following Proposition about Fourier integral operators.

\begin{prop}\label{L2prop}
Let
\begin{equation}
\label{L2prop0}
p = -m + 40\delta m, \qquad -\frac{m}{20} \leq q \leq 4\delta m+D^2,
\end{equation} 
with $\delta \leq 10^{-4}$.
For any $g \in L^2$ and $s \in [2^{m-1},2^{m+1}]$ define the operator
\begin{equation}
\label{L2prop1}
\begin{aligned}
 & T_{p,q} (g) (\xi) := \int_{\R^2} e^{is \Phi(\xi,\eta)} \varphi_{\leq p}(\Phi(\xi,\eta)) \varphi_q(\xi_1-\eta_1) \rho(\xi,\eta) \, g(\eta) \, d\eta,
\\
 & \Phi(\xi,\eta) = -L(\xi) + L(\xi-\eta) + L(\eta), \qquad L(x)=\frac{x_1}{\abs{x}^2},
\end{aligned}
\end{equation}
and assume that the symbol $\rho$ has the properties
\begin{equation}
\label{L2prop2.0}
\begin{aligned}
\mathrm{supp}(\rho) \subseteq \big\{ &(\xi,\eta)\in\R^2\times\R^2 \, : \, 2^{-A\delta m} \lesssim \abs{\xi},\abs{\eta} \lesssim 2^{A\delta m}, \,
  \\ 
&\abs{\xi-\eta},\abs{\xi-2\eta} \gtrsim 2^{-A\delta m}, \, \abs{2\xi-\eta} \gtrsim 2^{-7A\delta m}   \big\}
\end{aligned}
\end{equation}
for some absolute positive constant $A \leq 5$, 
and
\begin{align}
\label{L2prop2}
\big| D^\alpha_{(\xi,\eta)} \rho(\xi,\eta) \big| \lesssim 2^{\abs{\alpha}(m/2 + 60\delta m)} 2^{20\delta m}, \quad \abs{\alpha}\geq 0.
\end{align}
Then $T_{p,q}$ satisfies the operator bound
\begin{align}
\label{L2propconc}
 \norm{T_{p,q}}_{L^2 \rightarrow L^2} \lesssim 2^{-m - 100\delta m}.
\end{align}
\end{prop}

Before proceeding with the proof of this Proposition, let us explain how Proposition \ref{L2prop} implies the desired bound \eqref{estweightIIb}:

\begin{proof}[Proof of \eqref{estweightIIb} from Proposition \ref{L2prop}]
Without loss of generality we can assume $j_1\leq j_2$.
Then, according to our notation \eqref{B_q} and under the assumptions above, we can write
\begin{align*}
P_k \mathcal{B}_q(f_1,f_2) = \mathcal{F}^{-1} \int_\R \tau_m(s) \cdot \e_1 T_{p,q} (f_2) \, ds,
\end{align*}
where we let
\begin{align*}
\rho(\xi,\eta) = \mathfrak{m}^{k,k_1,k_2}(\xi,\eta) 
  \, \varphi_\ell(\xi-2\eta) \varphi_r(2\xi-\eta) \, \e_1^{-1} \widehat{f_1}(\xi-\eta).
\end{align*}

Using the a priori bound ${\| \what{f_1} \|} \lesssim 2^{-k_1}\e_1$ and the restriction on $\ul{j}$ in \eqref{restII}, it is easy to see that
the above $\rho(\xi,\eta)$ satisfies the hypotheses \eqref{L2prop2}.
Applying the conclusion \eqref{L2propconc} we can then estimate
\begin{align*}
\begin{split}
{\| P_k \mathcal{B}_q(f_1,f_2) \|}_{L^2}
    \lesssim \e_1 2^m {\|T_{p,q}\|}_{L^2 \rightarrow L^2} {\|f_2\|}_{L^2}
\lesssim \e_1 2^m \cdot 2^{-m - 100\delta m} \cdot \e_1 2^{-m +25 \delta m},
\end{split}
\end{align*}
which is sufficient to obtain \eqref{estweightIIb} in view of the restriction \eqref{restII}.
\end{proof}

The proof of Proposition \ref{L2prop} will be performed in the remainder of the paper and will conclude the proof of the Main Theorem \ref{thm:maintheo}.

\vskip5pt
\subsection{Proof of Proposition \ref{L2prop}}
To prove \eqref{L2propconc} we will use a $TT^\ast$ argument which is based on 
a suitable non-degeneracy property of the mixed Hessian of the phase $\Phi$.
In particular, it turns out to be crucial that we can integrate by parts along the direction parallel to the level sets of $\Phi$.
We subdivide the proof into a few steps: First, in Step 1 we describe a curvature quantity that gives a measure of the aforementioned non-degeneracy. Step 2 then sets up the $TT^\ast$ kernel and guides the subsequent splitting:
We either use smallness of sets to get the claimed kernel bounds (Step 3) or exploit the non-degeneracy via an iterated integration by parts (Step 4).

\medskip
\noindent
\textit{Step 1: The curvature quantity $\what{\Upsilon}$}.
In preparation for Step 2 let us define
\begin{equation}
\label{Upsilon}
\what{\Upsilon}(\xi,\eta) :=
  \nabla^2_{\xi,\eta}\Phi \left(\frac{\nabla^\perp_\xi\Phi}{\abs{\nabla_\xi\Phi}},\frac{\nabla^\perp_\eta\Phi}{\abs{\nabla_\eta\Phi}}\right)(\xi,\eta).
\end{equation}
We begin with the following algebraic lemma involving $\what{\Upsilon}$:

\begin{lemma}
Define $\Gamma$ and $\Theta$ as follows:
\begin{equation*}
\hat{\Upsilon}(\xi,\eta)=:\frac{\Gamma(\xi,\eta)}{\abs{\xi-\eta}^8 \abs{\nabla_\xi\Phi(\xi,\eta)}\abs{\nabla_\eta\Phi(\xi,\eta)}},
  \qquad \Phi(\xi,\eta)=:\frac{\Theta(\xi,\eta)}{\abs{\xi-\eta}^2}.
\end{equation*}
Then we have the identity
\begin{equation}
\label{eq:C-eq}
\frac{1}{2} \Gamma(\xi,\eta)-2\Theta(\xi,\eta) = 3(\xi_1-\eta_1).
\end{equation}
As a consequence, on the support of the operator $T_{p,q}$ the following bounds on $\what{\Upsilon}$ hold:
\begin{equation}
\label{eq:upssize}
2^{q-6A\delta m} \lesssim \big| \what{\Upsilon}(\xi,\eta) \big| \lesssim 2^{q+10A\delta m}.
\end{equation}
\end{lemma}

\begin{proof}
The identity \eqref{eq:C-eq} is obtained by a direct computation.

To verify \eqref{eq:upssize} notice that
\begin{equation*}
 |\what{\Upsilon}(\xi,\eta)| = \frac{\abs{\Gamma(\xi,\eta)}\abs{\xi}\abs{\eta}}{\abs{\xi-\eta}^4\abs{\xi-2\eta}\abs{\eta-2\xi}},
\end{equation*}
and therefore, because of the restrictions \eqref{L2prop2.0},
\begin{equation*}
2^{-6A\delta m} \abs{\Gamma(\xi,\eta)} \lesssim |\what{\Upsilon}(\xi,\eta)| \lesssim 2^{10A\delta m} \abs{\Gamma(\xi,\eta)}.
\end{equation*}
Now note that $\abs{\Theta(\xi,\eta)} \lesssim 2^p 2^{2A\delta m} \ll  2^q \approx \abs{\xi_1-\eta_1}$ by \eqref{L2prop0}-\eqref{L2prop1}.
Hence we can use \eqref{eq:C-eq} to deduce that $\abs{\Gamma} \approx 2^q$, and the conclusion follows.
\end{proof}

\noindent
\textit{Step 2: The $TT^\ast$ kernel}.
Notice that the support of $(T_{p,q} g)(\xi)$ is contained in the ball $|\xi| \lesssim 2^{4\delta m}$.
We decompose this ball into $O(2^{-2q+2(C_0+4)\delta m})$ balls of radius $R := 2^{q-C_0\delta m - D^3}$,
for some absolute constant $C_0 \in [50,150]$ to be determined below, depending on $A$.
If we denote by $\xi_0$ the center of any such small ball and let
\begin{equation*}
T_{p,q,\xi_0} (g) (\xi) := \varphi_{\leq R}(\xi-\xi_0) T_{p,q} (g) (\xi),
\end{equation*}
we see that the main bound \eqref{L2propconc} will follow provided we can show that for every $\xi_0 \in \R$,
\begin{align}
\label{L2propconc'}
 \norm{T_{p,q,\xi_0} T_{p,q,\xi_0}^\ast }_{L^2 \rightarrow L^2} \lesssim \big[ 2^{-m -100\delta m} \cdot 2^{2q - 2(C_0+4)\delta m} \big]^2.
\end{align}
Such a localization to a small ball in $\xi$ will allow us to better control several remainder terms in various Taylor expansions below.

Let us write
\begin{equation*}
T_{p,q,\xi_0}  T_{p,q,\xi_0}^\ast g (\xi) = \int_{\R^2} S_{p,q,\xi_0}(\xi,\xi^\prime) g(\xi^\prime) \, d\xi^\prime,
\end{equation*}
where the kernel is given by
\begin{align}
\label{eq:TT*kernel}
\begin{split}
S_{p,q,\xi_0}(\xi,\xi') 
  = 
  \varphi_{\leq R}(\xi-\xi_0) \varphi_{\leq R}(\xi^\prime-\xi_0)
  \int_{\R^2} e^{is[\Phi(\xi,\eta)-\Phi(\xi^\prime,\eta)]} \rho(\xi,\eta) \rho(\xi',\eta)
  \\
  \times \varphi_{q}(\xi_1-\eta_1) \, \varphi_{q}(\xi_1^\prime-\eta_1)
  \, \varphi_{\leq p}(\Phi(\xi,\eta)) \,\varphi_{\leq p}(\Phi(\xi^\prime,\eta)) \, d\eta.
\end{split}
\end{align}
Notice that on the support of this kernel we must have $|\xi-\xi^\prime| \leq 4R = 4 \cdot 2^{q-C_0\delta m -D^3}$.
Also recall that the symbol $\rho$ satisfies the properties \eqref{L2prop2.0}-\eqref{L2prop2}.
We will sometimes use the short-hand notation $S(\xi,\xi^\prime)$ for $S_{p,q,\xi_0}(\xi,\xi')$, dropping the indices where this creates no confusion.

To bound the relevant operator we will resort to an integration by parts in $\eta$ in the kernel \eqref{eq:TT*kernel} -- see Step 4.
Where this integration fails we will show how to gain from the smallness of the measure of the support of the kernel (Step 3).

The integration by parts will be performed through the following trivial identity:
\begin{equation}
\label{eq:ibpperp}
 e^{is[\Phi(\xi,\eta)-\Phi(\xi',\eta)]}
  = \frac{1}{is\mathcal{D}}\frac{\nabla_\eta^\perp\Phi(\xi,\eta)}{\abs{\nabla_\eta\Phi(\xi,\eta)}}\cdot\nabla_\eta e^{is[\Phi(\xi,\eta)-\Phi(\xi',\eta)]}
\end{equation}
with
\begin{equation}
\label{eq:defD}
 \mathcal{D} := \frac{\nabla_\eta^\perp\Phi(\xi,\eta)}{\abs{\nabla_\eta\Phi(\xi,\eta)}}\cdot\nabla_\eta[\Phi(\xi,\eta)-\Phi(\xi',\eta)].
\end{equation}
The choice of direction of intergration by parts is motivated by the roughness of the symbol in the integrand in \eqref{eq:TT*kernel}.
See also the identities \eqref{eq:nolossibp}-\eqref{eq:nolossibp'}.

To see the relevance of $\what{\Upsilon}$ defined in \eqref{Upsilon} we calculate
\begin{align*}
\begin{split}
\mathcal{D} & = \frac{\nabla_\eta^\perp\Phi(\xi,\eta)}{\abs{\nabla_\eta\Phi(\xi,\eta)}}\cdot\nabla_\eta[\Phi(\xi,\eta)-\Phi(\xi',\eta)]
\\
& = \frac{\nabla_\eta^\perp\Phi(\xi,\eta)}{\abs{\nabla_\eta\Phi(\xi,\eta)}}\cdot[\nabla^2_{\xi,\eta}\Phi(\xi,\eta)(\xi-\xi')]
  + O(\nabla^3_{\xi,\xi,\eta}\Phi(\xi,\eta)\abs{\xi-\xi'}^2).
\end{split}
\end{align*}
The fact that $\nabla_\xi\Phi$ does not vanish allows us write
\begin{equation*}
\xi-\xi'= a e_1 + b e_2, \qquad
e_1 := \frac{\nabla_\xi^\perp\Phi(\xi,\eta)}{\abs{\nabla_\xi\Phi(\xi,\eta)}}, \quad e_2 := \frac{\nabla_\xi\Phi(\xi,\eta)}{\abs{\nabla_\xi\Phi(\xi,\eta)}}.
\end{equation*}
We can thus decompose $\mathcal{D}$ as
\begin{equation*}
\mathcal{D} = a \what{\Upsilon}(\xi,\eta) +
  b\frac{\nabla_\eta^\perp\Phi(\xi,\eta)}{\abs{\nabla_\eta\Phi(\xi,\eta)}}
  \nabla^2_{\xi,\eta}\Phi(\xi,\eta)\frac{\nabla_\xi\Phi(\xi,\eta)}{\abs{\nabla_\xi\Phi(\xi,\eta)}}
  + O \big(\nabla^3_{\xi,\xi,\eta}\Phi(\xi,\eta)\abs{\xi-\xi'}^2 \big),
\end{equation*}
with $\what{\Upsilon}$ defined in \eqref{Upsilon} and satisfying the bounds \eqref{eq:upssize}.
In particular
\begin{equation}
\label{eq:Ddecomp'}
| \mathcal{D} | \geq |a| |\what{\Upsilon}(\xi,\eta)| -
  | b | \big| \nabla^2_{\xi,\eta}\Phi(\xi,\eta) \big| 
  - 2^D \big| \nabla^3_{\xi,\xi,\eta}\Phi(\xi,\eta) \big| |\xi-\xi'|^2.
\end{equation}

Observe that on the support of $S(\xi,\xi^\prime)$ we have
\begin{equation}
\label{eq:bvsxi}
\begin{aligned}
2^{p} \gtrsim \abs{\Phi(\xi,\eta)-\Phi(\xi',\eta)} &\gtrsim \abs{\nabla_\xi\Phi(\xi,\eta)\cdot(\xi-\xi')}
  - O \big( \abs{\nabla^2_\xi\Phi(\xi,\eta)}\abs{\xi-\xi'}^2 \big)
\\
&= \abs{b}\abs{\nabla_\xi \Phi(\xi,\eta)} - O \big( \abs{\nabla^2_\xi\Phi(\xi,\eta)}\abs{\xi-\xi'}^2 \big).
\end{aligned}
\end{equation}

\medskip
\noindent
\textit{Step 3: Case $\abs{b} \geq 2^{C_1\delta m + D} \abs{\xi-\xi^\prime}^2$, with 
$C_1 := 13A$}.
Using \eqref{eq:bvsxi}, 
$|\nabla_\xi\Phi(\xi,\eta)| \gtrsim 2^{-10A\delta m}$
and $|\nabla_{\xi\xi}^2 \Phi(\xi,\eta)| \lesssim 2^{3A\delta m}$, we deduce that
$|b| \lesssim 2^{p + 10A\delta m}$
and in particular that we must have
\begin{equation*}
 \abs{\xi-\xi^\prime}^2 \lesssim 2^{p}.
\end{equation*}
We now use Schur's test to show how this suffices to obtain \eqref{L2propconc'}.

More generally, let us assume that the support of $S(\xi,\xi^\prime)$ is contained in the set $|\xi-\xi^\prime| \leq L$.
Using Lemma \ref{lem:S+set}\eqref{it:Schur_p1},
the lower bounds 
$|\nabla_\xi\Phi(\xi,\eta)| \gtrsim 2^{-10A\delta m}$ 
and
$|\nabla_\eta\Phi(\xi,\eta)| \gtrsim 2^{-4A\delta m}$
that hold on the support of $\rho(\xi,\eta)$, see \eqref{L2prop2.0} and \eqref{nabla_xiPhi}, we can then estimate
\begin{align}
\label{L2propSchur}
\begin{split}
& \int_{\R^2} \abs{S(\xi,\xi')} \chi_{\{\abs{\xi-\xi'}\leq L \}} d\xi
\\
& \quad \lesssim \iint_{\R^2\times\R^2} \varphi_{\leq p}(\Phi(\xi,\eta)) |\rho(\xi,\eta)|
  \, \varphi_{\leq p}(\Phi(\xi',\eta)) |\rho(\xi^\prime,\eta)| \, \chi_{\{|\xi-\xi'| \lesssim L \}} d\eta d\xi
\\
& \quad \lesssim \int_{\R^2} \varphi_{\leq p}(\Phi(\xi',\eta)) |\rho(\xi^\prime,\eta)|
  \left[ \int_{\R^2} \varphi_{\leq p}(\Phi(\xi,\eta)) |\rho(\xi,\eta)| \, \chi_{\{\abs{\xi-\xi'}\lesssim L \}} d\xi \right] d\eta
\\
& \quad \lesssim \int_{\R^2} \varphi_{\leq p}(\Phi(\xi',\eta)) |\rho(\xi^\prime,\eta)| 
  \left[ 2^{p} \cdot 2^{(10A+20)\delta m} \cdot L \right] d\eta
\\
& \quad \lesssim 2^{2p} \cdot 2^{(14A+40)\delta m} \cdot L.
\end{split}
\end{align}
By symmetry
a similar bound also holds when exchanging the roles of $\xi$ and $\xi^\prime$.
Using this estimate with $L=2^{p/2}$, we see that \eqref{L2propconc'} follows from Schur's test 
since, under our assumptions, $(5/2)p+(14A+40)\delta m$ is less than $-2m-200\delta m + 4q - 4(C_0+4)\delta m$, as required.

\medskip
\noindent
\textit{Step 4: Case $\abs{b} \leq 2^{C_1\delta m + D} \abs{\xi-\xi^\prime}^2$}.
In this case we have $|b| \leq 2^{-D} |\xi-\xi^\prime|$, provided we choose $C_0\geq C_1+4$.
Therefore $|a| \geq (1/2) |\xi-\xi^\prime|$. 
Then we must also have
\begin{equation*}
2^q \abs{a} \geq 2^{C_0\delta m + D^2} \abs{\xi-\xi^\prime}^2,
\end{equation*}
since $\abs{\xi-\xi^\prime} \leq 4\cdot2^{q-C_0\delta m-D^3}$ on the support of the kernel.
From \eqref{eq:upssize} we know that $|\hat{\Upsilon}| |a| \geq 2^{q-6A\delta m -D} \abs{a}$,
and since we also have
\begin{align*}
|b| |\nabla^2_{\xi,\eta}\Phi(\xi,\eta)| + 2^D |\nabla^3_{\xi,\xi,\eta}\Phi(\xi,\eta)| |\xi-\xi^\prime|^2
  \leq 2^{(C_1+3A) \delta m + 2D} |\xi-\xi^\prime|^2,
\end{align*}
we can choose $C_0 \geq C_1 + 9A = 22A$,  
and invoke \eqref{eq:Ddecomp'} to deduce
\begin{equation*}
 \abs{\mathcal{D}} \gtrsim 2^{q-6A\delta m} \abs{a}.
\end{equation*}

Notice that we can also assume that $\abs{a} \gtrsim 2^{-3m/10}$, for otherwise $\abs{a}\approx\abs{\xi-\xi^\prime} \lesssim 2^{-3m/10}$ and the bound \eqref{L2propSchur} would give us
\begin{equation*}
 \int_\xi \abs{S(\xi,\xi')} \chi_{\{\abs{\xi-\xi'}\leq 2^{-3m/10} \}} d\xi \lesssim 2^{2p} \cdot 2^{(14A+40)\delta m} \cdot 2^{-3m/10},
\end{equation*}
so that \eqref{L2propconc'} would follow via Schur's test as above.

We now claim that an iterated integration by parts yields
\begin{align}
\label{L2propIBPbound}
 \abs{S(\xi,\xi^\prime)} \lesssim 2^{40\delta m}
  \Big[ 2^{-m} \abs{\mathcal{D}}^{-1} \max \big\{ 2^{\frac{m}{2} + 60\delta m}, \, 2^{-q}, \, \abs{\mathcal{D}}^{-1} 2^{(\frac{2}{N}+1)A\delta m},
  \, 2^{-p} \abs{\mathcal{D}} \big\} \Big]^M,
\end{align}
for any positive integer $M$.
Since $|\mathcal{D}| \gtrsim 2^{-\frac{2m}{5}}$, $p \geq -m + 40\delta m$ and $q \geq -\frac{m}{20}$, 
this bound clearly suffices to obtain \eqref{L2propconc'}.

To prove \eqref{L2propIBPbound}, we integrate by parts in $\eta$ in the integral \eqref{eq:TT*kernel} using the identities \eqref{eq:ibpperp}--\eqref{eq:defD}:
For notational convenience, we rewrite them here as
\begin{equation*}
\begin{alignedat}{3}
& e^{is\Psi} = \frac{1}{is} \mathcal{X} e^{is\Psi},
\qquad &\Psi(\xi,\xi',\eta) := \Phi(\xi,\eta)-\Phi(\xi',\eta), &
\\
& \mathcal{X}(\xi,\eta) := \frac{1}{\mathcal{D}} \mathcal{V} \cdot \nabla_\eta,
\qquad &\mathcal{X}^T(\xi,\eta) := \div_\eta \left(\frac{1}{\mathcal{D}} \mathcal{V}\, \cdot \right),
\qquad &\mathcal{V} := \frac{\nabla_\eta^\perp\Phi(\xi,\eta)}{\abs{\nabla_\eta\Phi(\xi,\eta)}}.
\end{alignedat}
\end{equation*}
Integrating by parts $M$ times will then give
\begin{align}
\label{L2propIBPest}
\begin{split}
 \abs{S(\xi,\xi^\prime)} \lesssim \int_{\R^2}
  2^{-mM} \Big| 
  (\mathcal{X}^T)^M \big[ \rho(\xi,\eta) \rho(\xi',\eta)
  \varphi_{q}(\xi_1-\eta_1) \, \varphi_{q}(\xi_1^\prime-\eta_1)
  \\ \times \varphi_{\leq p}(\Phi(\xi,\eta)) \,\varphi_{\leq p}(\Phi(\xi^\prime,\eta))\big]  \Big| \, d\eta.
\end{split}
\end{align}
%
Let us now analyze the various terms that arise in \eqref{L2propIBPest}:

\setlength{\leftmargini}{2.0em}
\begin{itemize}

\item[a.] When $\div_\eta\mathcal{V}$ hits the symbol $\rho(\xi,\eta)\rho(\xi^\prime,\eta)$ this produces a factor growing at most 
$2^{\frac{m}{2} + 60\delta m}$
in view of the assumption \eqref{L2prop2}. This is accounted for by the first term in the curly brackets in \eqref{L2propIBPbound}.

\item[b.] The terms that arise when $\div_\eta\mathcal{V}$ hits the cutoff $\varphi_{q}(\xi_1-\eta_1) \varphi_{q}(\xi^\prime_1 - \eta_1)$
  are bounded by $2^{-q}$.

\item[c.] To deal with the terms when $\div_\eta\mathcal{V}$ hits the denominator $\mathcal{D}$, 
it suffices to observe that on the support of the kernel,
\begin{equation*}
\big| D_\eta^{\alpha} \mathcal{D}(\xi,\eta) \big| \lesssim 2^{(2 + |\alpha|)A\delta m}.
\end{equation*}

\item[d.] For the term arising when $\div_\eta\mathcal{V}$ hits the cutoff $\varphi_{\leq p}(\Phi(\xi^\prime,\eta)) \varphi_{\leq p}(\Phi(\xi,\eta))$,
first notice that by construction 
\begin{equation}\label{eq:nolossibp}
 \mathcal{V} \cdot \nabla_\eta \varphi_{\leq p}(\Phi(\xi,\eta)) = 0.
\end{equation}
Moreover, we can calculate
\begin{equation}
\label{eq:nolossibp'}
\begin{aligned}
\mathcal{V}(\xi,\eta)\cdot\nabla_\eta (\varphi_{p}(\Phi(\xi',\eta)))
&= \mathcal{V}(\xi,\eta) \cdot \nabla_\eta\Phi(\xi',\eta) 2^{-p} (\varphi^\prime)_{p}(\Phi(\xi',\eta))
\\
&= -\mathcal{D} (\xi,\eta) \, 2^{-p} (\varphi^\prime)_{p}(\Phi(\xi',\eta)).
\end{aligned}
\end{equation}
We then see that this is accounted for by the last term in the curly brackets in \eqref{L2propIBPest}.
\end{itemize}

\noindent
This concludes the proof of \eqref{L2propIBPbound} and 
Proposition \ref{L2prop}. The Main Theorem \ref{thm:maintheo} follows. $\hfill \Box$

\vskip20pt
\section{Useful Lemmata}\label{sec:aux}

\vskip10pt
\subsubsection*{A Schur Lemma}
We demonstrate here some bounds for integral operators defined through kernels with localizations. 
These bounds derive from the set size restrictions brought about by localizations.
We first recall the standard Schur's test:
\begin{lemma}
For a kernel $K: \R^2 \times \R^2 \rightarrow \R$, consider the corresponding operator 
\begin{align*}
(T_K f)(\xi) := \int_{\R^2} K(\xi,\eta) f(\eta) \, d\eta, 
\end{align*}
and assume that
\begin{align*}
\sup_{\xi\in\R^2} \int_{\R^2} |K(\xi,\eta)| \, d\eta \leq K_1, \qquad  
  \sup_{\eta\in\R^2} \int_{\R^2} |K(\xi,\eta)| \, d\xi \leq K_2. 
\end{align*}
Then
\begin{align*}
{\| T_K f \|}_{L^2} \lesssim \sqrt{K_1K_2} {\| f \|}_{L^2}.
\end{align*}
\end{lemma}

We will often apply the above lemma, and for this purpose define
\begin{align}
\label{Sch}
\norm{K}_{Sch} := \Big( \sup_\xi\int K(\xi,\eta)d\eta \Big)^{\frac{1}{2}} \Big( \sup_\eta\int K(\xi,\eta)d\xi \Big)^{\frac{1}{2}}.
\end{align}


\begin{lemma}
\label{lem:S+set}

\begin{enumerate}
  
\item\label{it:Schur_p1} Let $F:\R^2\to\R$ be smooth in a ball $B_R(z)\subset\R^2$, $z\in\R^2$, $R>0$. Then
  \begin{equation*}
   \int_{B_R(z)}\varphi_{\leq\lambda}(F(x))\varphi_{\geq\mu}(\nabla F(x)) \, dx\leq 2^{-\mu}2^\lambda R.
  \end{equation*}

\item\label{it:Schur_p2} Consider an integral operator given by the kernel 
  $$K(\xi,\eta):=\varphi_p(\Phi(\xi,\eta))\varphi_\ell(\xi-2\eta)\varphi_r(\eta-2\xi)\varphi_k(\xi)\varphi_{a}(\xi-\eta)\varphi_{b}(\eta),$$
  where $\Phi$ is the phase in \eqref{defPhi}.
  Then we have the bound
  \begin{equation}\label{eq:Schur_p2}
  \norm{K}_{Sch} \lesssim 2^{p+\frac{1}{2}(k+b-\ell-r)+2a}2^{\frac{1}{2}\min\{\ell,r,a,b\}+\frac{1}{2}\min\{\ell,r,k,a\}},
  \end{equation}
  so that, in particular,
  \begin{equation*}
  \norm{K}_{Sch} \lesssim 2^{p+\frac{1}{2}(k+b+2a)}.
  \end{equation*}
  As a consequence, we also see that if $\min\{k,\ell\} \leq \max\{a,b\}-10$, 
  then, for $K^\ell(\xi,\eta) := \phi_p(\Phi(\xi,\eta)) \phi_\ell(\xi-2\eta) \phi_k(\xi) \phi_{a}(\xi-\eta) \phi_{b}(\eta)$ we have the bound
  \begin{equation}\label{eq:Schur_p3}
  {\| K^\ell \|}_{Sch} \lesssim 2^{p+\frac{1}{2}(k+b-\ell+3a)} 2^{\frac{1}{2}\min\{\ell,a,b\} + \frac{1}{2}\min\{\ell,k,a\}},
  \end{equation}

\end{enumerate}

\end{lemma}

\begin{proof}
Point \eqref{it:Schur_p2} 
is a consequences of \eqref{it:Schur_p1} 
and the formulas for the gradient of $\Phi$ in \eqref{nabla_xiPhi}, so we start by demonstrating \eqref{it:Schur_p1}.

\smallskip
\noindent
\emph{Proof of \eqref{it:Schur_p1}.} Notice that $\{x\in\R^2:\;\abs{\nabla F(x)}\geq 2^\mu\}\subset A_\mu^1 \cup A_\mu^2$, 
where $A_\mu^i:=\{x\in\R^2:\;\abs{\partial_{x_i}F(x)}\geq 2^{\mu-1}\}$. 
Hence on $B_R(z)\cap A_\mu^1$ a well-defined change of variables is given by $(y_1,y_2)=Y(x):=(F(x_1,x_2),x_2)$. 
This change of variables has Jacobian determinant equal to $\abs{\partial_{x_1}F}\gtrsim 2^{\mu}$, so we have
\begin{equation*}
 \begin{aligned}
 \int_{B_R(z)\cap A_1^\mu} \varphi_{\leq\lambda}(F)\varphi_{\geq\mu}(\nabla F)(x) \, dx 
  &\lesssim2^{-\mu}\int_{Y(B_R(z))}\varphi_{\leq\lambda}(F)\varphi_{\geq\mu}(\nabla F)(Y^{-1}(y)) \, dy
  \\
  &\lesssim 2^{-\mu} \int_{\abs{y_2-z_2}\leq R} \varphi_{\leq\lambda}(y_1) \, dy\leq 2^{-\mu} 2^{\lambda}R.
 \end{aligned}
 \end{equation*}
 Exchanging the roles of $x_1$ and $x_2$, in complete analogy we deduce the same bound for 
 $$\int_{B_R(z)\cap A_2^\mu} \varphi_{\leq\lambda}(F)\varphi_{\geq\mu}(\nabla F)\, dx,$$ thus proving the first claim.

\smallskip
\noindent
\emph{Proof of \eqref{it:Schur_p2}.} We estimate the two integrals in \eqref{Sch},
for each of which it will suffice to appropriately apply \eqref{it:Schur_p1}.
To this end, notice that with the localizations in $K(\xi,\eta)$ we have, see \eqref{nabla_xiPhi}, 
$$ \abs{\nabla_\eta\Phi} = \frac{\abs{\xi}\abs{\xi-2\eta}}{\abs{\eta}^2\abs{\xi-\eta}^2}\approx 2^{k+\ell}2^{-2a-2b},
  \qquad \abs{\nabla_\xi\Phi} = \frac{\abs{\eta}\abs{\eta-2\xi}}{\abs{\xi}^2\abs{\xi-\eta}^2}\approx 2^{b+r}2^{-2k-2a}$$ 
and $\Phi$ is smooth in the domains of integration.

Furthermore, for fixed $\xi$ there exist $\xi_0$ and $R \lesssim \min\{2^\ell,2^r,2^{a},2^{b}\}$ 
such that the domain of the integral in $\eta$ is contained in the ball $B_R(\xi_0)$. We then invoke (1) to obtain
\begin{equation*}
\begin{aligned}
  \int_{\R^2} K(\xi,\eta) \, d\eta &\leq 
  \int_{B_R(\xi_0)}\varphi_p(\Phi(\xi,\eta)) \varphi_{2^{k+\ell-2a-2b}}\big(2^{-10}\nabla_\eta\Phi(\xi,\eta)\big) \, d\eta
  \\ &\lesssim 2^{p} 2^{-k-\ell+2a+2b} 2^{\min\{\ell,r,a,b\}}.
\end{aligned}
\end{equation*}
Similarly, for fixed $\eta$  there exists $\eta_0$ such that the domain of the integral in $\xi$ is included in a ball of 
center $\eta_0$ and radius $R \lesssim \min\{2^\ell,2^r,2^{k},2^{a}\}$, which promptly yields
\begin{equation*}
\int_{\R^2} K(\xi,\eta) \, d\xi \lesssim 2^{p} 2^{-b - r + 2k + 2a} 2^{\min\{\ell,r,k,a\}}.
\end{equation*}
Combining these gives the claim \eqref{eq:Schur_p2}.
The bound \eqref{eq:Schur_p3} follows since for $\min\{k,\ell\} \leq \max\{a,b\} - 10$ one has $|r-\max\{a,b\}| \leq 5$.

\end{proof}

\vskip10pt
\subsubsection*{H\"older Type Estimates and Integration by Parts Lemmas} 
For simplicity of notation we define the following class of multipliers:
\begin{equation}
\label{defSinfty}
\begin{split}
& S^\infty := \{m: (\R^2)^2 \to \mathbb{C} : \,m  \quad \text{continuous and} \quad {\| m \|}_{S^\infty} := \|\mathcal{F}^{-1} m \|_{L^1} < \infty \}.
\end{split}
\end{equation}
As we will often localize in frequency space we define, for any symbol $m$,
\begin{align}
\label{mloc}
m^{k, k_1, k_2}(\xi,\eta) := \varphi_{[k-2,k+2]}(\xi)\varphi_{[k_1-2,k_1+2]}(\xi-\eta)\varphi_{[k_2-2,k_2+2]}(\eta) m(\xi,\eta),
\end{align}
see the notation in Section \ref{sec:setup}.
Here is a basic lemma about $S^\infty$ symbols that we will often use:

\begin{lemma}\label{lem:Holdertype}

\setlength{\leftmargini}{1.5em}
\begin{itemize}

\item[(i)] We have $S^\infty \hookrightarrow L^\infty(\R^2\times\R^2)$. If $m, m'\in S^\infty$ then $m \cdot m'\in S^\infty$ and
\begin{equation*}
\|m\cdot m'\|_{S^\infty}\lesssim \|m\|_{S^\infty}\|m'\|_{S^\infty}.
\end{equation*}
Moreover, if $m \in S^\infty$, $A: \R^2 \to \R^2$ is a linear transformation, $v\in \R^2$, and $m_{A,v}(\xi,\eta) := m(A(\xi,\eta)+v)$, then
\begin{equation*}
\|m_{A,v}\|_{S^\infty} = \|m\|_{S^\infty}.
\end{equation*}

\item[(ii)] For $m \in S^\infty$, consider the bilinear operator $T_m:\mathcal{S}(\R^2)\times\mathcal{S}(\R^2)\to\mathcal{S'}(\R^2)$ 
defined by
\begin{align*}
T_m(f,g)(\xi) := \mathcal{F}^{-1}\int m(\xi,\eta) \what{f}(\xi-\eta) \what{g}(\eta)d\eta.
\end{align*}
Then, for all $1\leq p,q,r\leq\infty$ satisfying the H\"older relation $\frac{1}{r}=\frac{1}{p}+\frac{1}{q}$, we have
\begin{align*}
{\| T_m(f,g) \|}_{L^p} \lesssim {\| m \|}_{S^\infty} {\| f \|}_{L^p} {\| g \|}_{L^q}.
\end{align*}

\end{itemize}

\end{lemma}

\begin{proof}
The properties in (i) follow directly from the definition \eqref{defSinfty}.
A direct computation unwinding the Fourier transforms shows that
\begin{align*}
T_m(f,g)(x)&=\int_\xi e^{ix\xi}\int_\eta m(\xi,\eta)\hat{f}(\xi-\eta)\hat{g}(\eta) \, d\eta d\xi\\
   & =\int_y\int_z f(x-z)g(x-y-z)\check{m}(z,y) \, dydz,
\end{align*}
 from which the claim follows directly.
\end{proof}

\medskip
We state next a useful lemma, which 
allows us to use H\"older type bounds when we integrate by parts in time. 

\begin{lemma}\label{lem:H+ibpt}
Assume $t \approx 2^m$ for some $m\in\N$, and $p \geq -m+2\delta m$.
For $\rho \in S^\infty$, with ${\|\rho\|}_{S^\infty} \leq 1$, consider a bilinear operator of the form
\begin{equation*}
\underline{B}_p(v, w)(\xi) := \varphi_{\leq 10 m}(\xi) \int_{\R^2} e^{it\Phi(\xi,\eta)}
  \chi(2^{-p}\Phi(\xi,\eta)) \rho(\xi,\eta)\,\widehat{v}(\xi-\eta) \widehat{w}(\eta) \, d\eta,
\end{equation*}
where $\chi$ is a Schwartz function.
Then, for any $1/p+1/q = 1/2$,
\begin{equation*}
\norm{ \underline{B}_p(v,w)}_{L^2} \lesssim
  \Big( \sup_{\abs{s} \leq 2^{-p}2^{\delta m}} {\|e^{i(t+s)L} v\|}_{L^p} {\|e^{i(t+s)L}w \|}_{L^q} 
  + {\| v \|}_{L^2} {\| w \|}_{L^2} 2^{-10m} \Big).
\end{equation*}
\end{lemma}

\begin{proof}
Let us use
\begin{equation*}
\chi(2^{-p}\Phi(\xi,\eta)) = c \int_{\R} e^{iz 2^{-p}\Phi(\xi,\eta)} \check{\chi}(z) \, dz
\end{equation*}
to write
\begin{equation*}
\underline{B}_p(v,w) = c \int_{\R^2} \left( \int_{\R} e^{i(2^{-p}z + t)\Phi(\xi,\eta)} \check{\chi}(z) \, dz \right) 
  \rho(\xi,\eta) \hat{v}(\xi-\eta) \hat{w}(\eta) \, d\eta.
\end{equation*}
Using the rapid decay $\abs{\check{\chi}}\leq (1+\abs{z})^{-M}$, for $M$ large enough,
we can estimate the contribution from the region $\abs{z}\geq 2^{\delta m}$ as
\begin{align*}
  \norm{\int_{\R^2} \Big( \int_{\abs{z} \geq 2^{\delta m}} e^{i(2^{-p}z + t)\Phi(\xi,\eta)}
  \check{\chi}(z) \, dz \Big) \varphi_{\leq 10 m}(\xi) \rho(\xi,\eta)
  \widehat{v}(\xi-\eta) \widehat{w}(\eta) \, d\eta }_{L^2_\xi}
\\
  \lesssim 2^{10m} 2^{-\delta M m} \norm{ v }_{L^2} \norm{ w }_{L^2} 
  \lesssim 2^{-10m} \norm{ v }_{L^2} \norm{ w }_{L^2}.
\end{align*}

We are now left with estimating
\begin{align*}
\norm{\int_{\R^2} \int_{\abs{z} \leq 2^{\delta m}} \!\! \check{\chi}(z) \, \varphi_{\leq 10 m}(\xi) \rho(\xi,\eta) \,
  e^{i(2^{-p}z + t)L(\xi-\eta)}  \widehat{v}(\xi-\eta) e^{i(2^{-p}z + t)L(\eta)}\widehat{w}(\eta) \, d\eta dz }_{L^2_\xi}
\\
\lesssim \sup_{\abs{z} \leq 2^{\delta m}}
  \norm{\int_{\R^2} \rho(\xi,\eta)
  e^{i(2^{-p}z + t)L(\xi-\eta)}  \widehat{v}(\xi-\eta) e^{i(2^{-p}z + t)L(\eta)}\widehat{w}(\eta) \, d\eta }_{L^2_\xi},
\end{align*}
which by virtue of Lemma \ref{lem:Holdertype} and ${\|\rho\|}_{S^\infty}\leq 1$ is bounded by
\begin{equation*}
\sup_{\abs{z} \leq 2^{\delta m}} \norm{ e^{i(t+2^{-p}z)L} v }_{L^p} \norm{ e^{i(2^{-p}z + t)L} w }_{L^q}.
\end{equation*}
The desired conclusion follows.
\end{proof}

\medskip
Here is a basic integration by parts lemma:
\begin{lemma}\label{lemIBP0}
Assume that $\epsilon \in(0,1)$, $\epsilon K\geq 1$, $M \geq 1$ is an integer, and $F,g \in C^M (\R^n)$.
Assume also that $F$ is real-valued and satisfies
\begin{equation*}
|\nabla F| \geq \mathbf{1}_{\supp(g)}, \qquad \big| D^{\alpha} F \big| \lesssim_M \epsilon^{1-|\alpha|} \quad \forall \,\, 2 \leq |\alpha| \leq M.
\end{equation*}
Then
\begin{align}
\label{IBP02}
\Big| \int_{\R^n} e^{iKF} g \, dx \Big| \lesssim \frac{1}{(\epsilon K)^M} \sum_{|\alpha|\leq M} \epsilon^{|\alpha|} {\| D^\alpha g \|}_{L^1}
\end{align}
\end{lemma}

The proof is a fairly straightforward integration by parts argument, see Lemma 5.4 in \cite{IoPau1}.

\bibliographystyle{amsplain}
\bibliography{allrefs.bib}

\end{document}